\documentclass[onefignum,onetabnum]{siamart220329}
\usepackage{amsmath}
\usepackage{amssymb}
\usepackage[bmargin=2.5cm,lmargin=2.5cm,rmargin=3.5cm,tmargin=2.5cm,asymmetric]{geometry}
\usepackage{mathrsfs}
\usepackage{mathtools}
\usepackage{amscd}
\usepackage[all]{xy}
\SelectTips{cm}{}
\usepackage{multirow}
\usepackage{stmaryrd}
\usepackage{enumitem}
\usepackage{textgreek}
\usepackage{hyperref}
\usepackage{cleveref}

\newcommand{\rev}[1]{{\color{black}#1}}
\usepackage{tikz}
\tikzstyle{nodes1} = [circle, rounded corners, minimum width=1cm, minimum height=1cm,text centered, draw=black, fill=white!30]
\tikzstyle{arrow} = [thick,->,>=stealth]

\usepackage{xargs}                      
\usepackage{xcolor}  

\usepackage{algorithm}
\usepackage{algpseudocode}

\usepackage[colorinlistoftodos,prependcaption,textsize=tiny]{todonotes}
\newcommandx{\unsure}[2][1=]{\todo[linecolor=red,backgroundcolor=red!25,bordercolor=red,#1]{#2}}
\newcommandx{\change}[2][1=]{\todo[linecolor=blue,backgroundcolor=blue!25,bordercolor=blue,#1]{#2}}
\newcommandx{\info}[2][1=]{\todo[linecolor=OliveGreen,backgroundcolor=OliveGreen!25,bordercolor=OliveGreen,#1]{#2}}
\newcommandx{\improvement}[2][1=]{\todo[linecolor=Plum,backgroundcolor=Plum!25,bordercolor=Plum,#1]{#2}}
\newcommandx{\thiswillnotshow}[2][1=]{\todo[disable,#1]{#2}}

\newcommand{\re}{\mathbb{R}}
\newcommand{\mR}{\mathbb{R}}

\newcommand{\N}{\mathbb{N}}

\newcommand{\lmd}{\lambda}

\def\af{\alpha}

\newcommand{\reff}[1]{(\ref{#1})}

\newcommand{\qmod}[1]{\mbox{QM}[#1]}

\newcommand{\st}{\mathit{s.t.}}

\newcommand{\bdes}{\begin{description}}
	\newcommand{\edes}{\end{description}}
\newcommand{\bal}{\begin{align}}
	\newcommand{\eal}{\end{align}}
\newcommand{\bnum}{\begin{enumerate}}
	\newcommand{\enum}{\end{enumerate}}
\newcommand{\bit}{\begin{itemize}}
	\newcommand{\eit}{\end{itemize}}
\newcommand{\bea}{\begin{eqnarray}}
	\newcommand{\eea}{\end{eqnarray}}
\newcommand{\be}{\begin{equation}}
	\newcommand{\ee}{\end{equation}}
\newcommand{\baray}{\begin{array}}
	\newcommand{\earay}{\end{array}}
\newcommand{\bsry}{\begin{subarray}}
	\newcommand{\esry}{\end{subarray}}
\newcommand{\bca}{\begin{cases}}
	\newcommand{\eca}{\end{cases}}
\newcommand{\bcen}{\begin{center}}
	\newcommand{\ecen}{\end{center}}
\newcommand{\bbm}{\begin{bmatrix}}
	\newcommand{\ebm}{\end{bmatrix}}
\newcommand{\bmx}{\begin{matrix}}
	\newcommand{\emx}{\end{matrix}}
\newcommand{\bpm}{\begin{pmatrix}}
	\newcommand{\epm}{\end{pmatrix}}
\newcommand{\btab}{\begin{tabular}}
	\newcommand{\etab}{\end{tabular}}

\def\<#1,#2>{\langle #1,#2\rangle}
\def\<#1>{\langle #1\rangle}

\newtheorem{prop}[theorem]{Proposition}

\newtheorem{cor}[theorem]{Corollary}

\makeatletter
\newcommand{\extp}{\@ifnextchar^\@extp{\@extp^{\,}}}

\def\@extp^#1{\mathop{\bigwedge\nolimits^{\!#1}}}
\makeatother

\renewcommand{\thesection}{\arabic{section}}
\renewcommand{\theequation}{\thesection.\arabic{equation}}
\numberwithin{equation}{section}

\begin{document}
	
	\title{On the complexity of matrix  Putinar's   Positivstellens{\"a}tz}
	
	\author{Lei Huang\thanks{Department of Mathematics,
			University of California San Diego (\email{leh010@ucsd.edu})}
	}

	\maketitle
	
	\begin{abstract}
		This paper studies the complexity of matrix Putinar's Positivstellens{\"a}tz on the semialgebraic  set that is given by the polynomial matrix inequality.  \rev{When the quadratic module generated by the constrained polynomial matrix is Archimedean},  we prove  a polynomial bound on the degrees of  terms appearing  in the representation of matrix Putinar's Positivstellens{\"a}tz.
		Estimates on the exponent and  constant are  given.   
		As a byproduct,  a polynomial bound on the  convergence rate of   matrix sum-of-squares relaxations is obtained, which resolves an open question raised by Dinh and  Pham.  When the constraining set is unbounded, we also prove a similar bound for the matrix version of Putinar--Vasilescu's Positivstellens{\"a}tz by exploiting homogenization techniques. 
	\end{abstract}
	
	

	\section{Introduction}
	
	Optimization with polynomial matrix inequalities has broad applications \cite{cgcon,FF20,hdlb,hdlb2,HDLJ,Ichihara,cscon,schhol,tzkh,yzgf}.
	A fundamental question  to solve it globally is  how to  efficiently characterize polynomial matrices that are  positive semidefinite on the semialgebraic set described by the  polynomial matrix inequality. To be more specific, given an $\ell$-by-$\ell$ symmetric polynomial matrix $F(x)$   in $x:=(x_1,\dots,x_n)$ and a semialgebraic set $K$, 
	how can we  certify  that $F\succeq 0$ on $K$?  Here, the set $K$ is given as 
	\be \label{def:K}
	K=\{x\in \mR^n \mid G(x)\succeq 0\},
	\ee
	where  $G(x)$  is an $m$-by-$m$ symmetric polynomial  matrix.
	
	When $F(x)$ is a scalar polynomial (i.e., $\ell=1$) and  $G(x)$ is  a diagonal matrix whose diagonal entries are polynomials $g_1(x)$, $\dots$, $g_m(x)$,  $K$ can be equivalently described as 
	\be \label{scalar:K}
	K=\{x\in \mR^n \mid g_1(x)\geq 0,\dots,g_m(x)\geq 0\}.
	\ee
	Then, it reduces to the classical Positivstellens{\"a}tz in real algebraic geometry and polynomial optimization.   A polynomial $p$ is said to be a sum of squares (SOS) if
	$ p=p_1^2+\dots+p_t^2$ for polynomials $p_1,\dots,p_t$. 
	Schm{\"u}dgen \cite{sk91} proved that if $K$ is compact and $F>0$ on $K$, then there exist SOS polynomials $\sigma_{\alpha}$ such that
	\be \label{preorder}
	F=\sum\limits_{\alpha\in \{0,1\}^m}\sigma_{\alpha} g_1^{\alpha_1}\cdots g_m^{\alpha_m}.
	\ee
	\rev{When the quadratic module of scalar polynomials $\qmod{G}$ generated by $G$ is Archimedean (a condition slightly stronger than compactness, refer  to Section \ref{sec:qm})},  Putinar \cite{putinar1993positive} showed that if  $F>0$ on $K$, then there exist SOS polynomials $\sigma_{0},\sigma_{1},\dots,\sigma_{m}$ such that
	\be \label{quamod}
	F=\sigma_{0}+\sigma_{1} g_1\rev{+\cdots+}\sigma_{m} g_m,
	\ee
	which is referred to  as Putinar's Positivstellens{\"a}tz in the literature.

	The complexities of  the above scalar Positivstellens{\"a}tze have been studied extensively. Suppose that $K\subseteq (-1,1)^n$ and the degree of $F$ is $d$. For the degree $d$, denote 
	\[
	\mathbb{N}_d^n:=\left\{\alpha \in \mathbb{N}^n \mid \left.\alpha_1+\cdots+\alpha_n \leq d\right\}.\right.
	\] 
	For a polynomial  $p=\sum\limits_{\alpha \in \mathbb{N}_d^n}p_{\alpha}\frac{|\alpha|!}{\alpha_1!\cdots\alpha_n!}x^{\alpha}$, denote the norm  
	$
	\|p\|=\max\{\rev{|p_{\alpha}|}:\alpha \in \mathbb{N}_d^n\}.
	$ 
	It was shown in \cite{sm04} that if $\rev{F\geq F_{\min}>0}$ on $K$, then for 
	\[
	k\geq    c d^2\left(1+\left(d^2 n^d \frac{\|F\|}{F_{\min}}\right)^c\right),
	\]
	there exist SOS polynomials $\sigma_{\alpha} $  with   $\deg(\sigma_{\alpha} g_1^{\alpha_1}\cdots g_m^{\alpha_m})\leq k$  such that \reff{preorder} holds. Here, $c$ is a positive constant that depends on the constraining polynomials $g_1,\dots,g_m$.
	Based on the above bound, Nie and Schweighofer \cite{niesch} proved that if  $\qmod{G}$ is Archimedean,  there exists a   constant $c>0$, depending on  $g_1,\dots,g_m$, such that if
	\be \label{Put:deg}
	k \geq c \exp \left(\left(d^2 n^d \frac{\|F\|}{F_{\min}}\right)^c\right),
	\ee
	then  \reff{quamod} holds for some SOS polynomials $\sigma_0$, $\dots$, $\sigma_m $  with   $\deg(\sigma_0)\leq k$, $\deg(\sigma_i g_i)\leq k$ $(i=1,\dots,m)$. Recently,  Baldi and  Mourrain improved the exponent bound \reff{Put:deg} to a polynomial bound. It was shown in \cite{lbbm}  that  the same result holds  if 
	\be \label{Put:deg2}
	k\geq  \gamma d^{3.5nL}\left(\frac{\|F\|_{\max}}{F_{\min}}\right)^{2.5nL}.
	\ee
	In the above,   $ \gamma$ is a positive constant depending on $n$, $g_1,\dots,g_m$, $L$ is the  {\L}ojasiewicz exponent associated with $g_1,\dots,g_m$ and $\|F\|_{\max}$ is the max norm on the cube $[-1,1]^n$.   The  bound \reff{Put:deg2} has been further improved in \cite{bmp} by  removing the
	dependency of the exponent  on the dimension $n$ and providing   estimates on the constant $\gamma$.
	There  exist better degree bounds when the set $K$ is specified, see \cite{bfra,blls,FF20,slot,sllm1,sllm2}. We also refer to \cite{kmz} when $F,g_1,\dots,g_m$  admit the correlative sparse sparsity.

	Based on Putinar's Positivstellens{\"a}tz, Lasserre \cite{Las01} introduced  the    Moment-SOS hierarchy for solving    polynomial optimization  and proved its asymptotic convergence.   The above  degree bounds give estimates on the  convergence rates of the    Moment-SOS hierarchy directly.
	We also refer to  \cite{dKlLau11,huang2023,hny,hny2,mar06,nieopcd,Sch05}
	for the finite convergence of the  Moment-SOS
	hierarchy under some regular assumptions.

	When $F$, $G$ are general polynomial matrices,  a matrix version of Putinar's Positivstellens{\"a}tz was provided in \cite{kojm,schhol}. 
	\rev{Let $\mathbb{R}[x]:=\mathbb{R}\left[x_1, \ldots, x_n\right]$ be the ring of polynomials in $x:=\left(x_1, \ldots, x_n\right)$ with real coefficients.
	A polynomial
	matrix $S(x)$ is SOS if $S(x) = H(x)^TH(x)$ for some polynomial matrix $H(x)$.  The set of all $\ell$-by-$\ell$  SOS polynomial matrices in $x$ is denoted as $\Sigma[x]^{\ell}$.
	The  quadratic module of $\ell$-by-$\ell$  polynomial matrices generated
	by  $G(x)$  is the set
	\be \nonumber 
	\qmod{G}^{\ell}:=\left\{
	S+\sum_{i=1}^t P_i^TGP_i \mid t\in \mathbb{N},~ S \in \Sigma[x]^{\ell},~P_i \in \mathbb{R}[x]^{m\times \ell}
	\right\}.
	\ee
	 	The set $\qmod{G}^{\ell}$ is Archimedean if there exists a scalar $r>0$ such that $(r-\|x\|_2^2) \cdot I_{\ell} \in \mathrm{QM}[G]^{\ell}$.
	 	 When the quadratic module $\qmod{G}^{\ell}$ is Archimedean,}  it was proved in \cite{kojm,schhol} that if $F\succ 0$ on $K$, then we have that $F \in \qmod{G}^{\ell}$.
   For the  case that  $F$ is a polynomial matrix and $G$ is diagonal (i.e., $K$ is as in  \reff{scalar:K}), the exponential bound \reff{Put:deg} was generalized to the matrix case in \cite{helnie}. \rev{There are specific estimates on degree bounds when  $G$ is a simple set defined by scalar polynomial inequalities. We refer to \cite{FF20} for the sphere case and refer to \cite{sllm2} for the  boolean hypercube case.} When $G$ is a general polynomial matrix,  the complexity of   matrix  Putinar's Positivstellens{\"a}tz  is open, to the best of the author's knowledge.   We also refer to  \cite{cj,kism,sk09} for other types of SOS representations for polynomial matrices.

	\subsection*{Contributions}
	This paper studies the complexity of matrix Putinar's Positivstellens{\"a}tz.
	Throu-\\ghout the paper, we assume that 
	\[
	K\subseteq \{x\in\mR^n\mid 1-\|x\|_2^2\geq0\},
	\]
	unless otherwise specified.
	Denote the scaled simplex 
	\be \label{sca:simplex}
	\bar{\Delta}:=\left\{x \in \mathbb{R}^n \mid 1+x_1 \geq 0, \ldots, 1+x_n \geq 0, \sqrt{n}-x_1-\cdots-x_n \geq 0\right\}.
	\ee
	Note that the set $K \subset \bar{\Delta}$ in our settings.  The degree of the polynomial matrix $G$ is denoted by $d_{G}$, i.e.,
	\be \label{deg:G}
	d_{G}=\max\{\deg(G_{ij}),~1\leq i, j\leq m\}.
	\ee
	For a positive integer $m$,  denote 
	\be  \label{int:num:scapol}
	\theta(m)=\sum\limits_{i=1}^m\left(\prod_{k=m+1-i}^{m} \frac{k(k+1)}{2}\right).
	\ee
	Let  $\|F\|_B$  be the Bernstein norm  of $ F$ on the scaled simplex $\bar{\Delta}$ (see Section \ref{Matrix:poly} for details). 

	The following is our main result.  It generalizes the recent results of  scalar Putinar's Positivstellens{\"a}tz, as shown in  \cite{lbbm,bmp}, to the matrix case  and  provides the first degree bound for matrix Putinar's Positivstellens{\"a}tz.
	\begin{theorem}
		\label{intro:comp:matPut3}
		Let $K$ be as in \reff{def:K} with $m\geq 2$.	Suppose $F(x)$ is an $\ell$-by-$\ell$   symmetric  polynomial matrix of degree $d$, and $ (1-\|x\|_2^2)\cdot I_{\ell}\in \qmod{G}^{\ell}$.
		If $F(x)\succeq F_{\min} \cdot  \rev{I_{\ell}} \succ 0$ on $K$, then for
		\be \label{put3:degboun}
		k\geq  C \cdot  {8}^{7\eta}\cdot3^{6(m-1)} \theta(m)^3n^2d_{G}^{6}\kappa^{7}d^{14\eta}\left(\frac{\|F\|_B}{F_{\min}}\right)^{7\eta+3},
		\ee
		there exists an $\ell$-by-$\ell$  SOS matrix $S$  and  polynomial matrices $P_{i}\in \mR[x]^{m\times \ell}$ $(i=1,\dots,t)$   with 
		\[
		\deg(S)\leq k,~ \deg(P_i^TGP_i)\leq k~(i=1,\dots,t), 
		\]
		such that
		\be  \label{mat:matr}
		F=S+P_1^TGP_1+\cdots +P_t^TGP_t.
		\ee
		In the above,  $\kappa, \eta$ are  respectively the {\L}ojasiewicz constant and exponent  in \reff{def:Loj:G4}, and  $C>0$ is a constant  independent of $F$, $G$. \rev{Furthermore, the {\L}ojasiewicz exponent can be estimated as 
			\[
			\eta\leq (3^{m-1}d_G+1)(3^{m}d_G)^{n+\theta(m)-2}.
			\]
		}
	\end{theorem}

	The positivity certificate \reff{mat:matr} can be applied to construct matrix SOS relaxations for solving polynomial matrix optimization \cite{hdlb,schhol}.    As a byproduct of  Theorem \ref{intro:comp:matPut3}, we can obtain a polynomial bound on the convergence rate of matrix  SOS relaxations (see Theorem \ref{matrate}), which  resolves an open question raised in \cite{dihu}. 
	
	\rev{
		When $F(x)$ is a scalar polynomial and  $G(x)$ is   diagonal with diagonal entries $g_1(x), \dots, g_m(x)\in \mR[x]$, it was shown in \cite{bmp}  that if $F\geq F_{\min}>0$ on $K$, then $f$ admits a representation as in \reff{quamod} for 
		\[
		k\geq C\cdot 2^{\eta} \cdot n^2 m d_G^6  d^{2\eta} \kappa^7 \left(\frac{\|F\|_B}{F_{\min}}\right)^{7\eta+3},
		\]
		where $\kappa, \eta$ are respectively the {\L}ojasiewicz constant and exponent, and  $C>0$ is a constant  independent of $F$, $G$.
	Comparing this with  the bound as in Theorem \ref{intro:comp:matPut3}, one important   additional factor is  $3^{6(m-1)} \theta(m)^3$.  This comes  from  the   estimates on the quantity and degree bounds of scalar  polynomials  that can be used to describe the polynomial matrix inequality $G(x)\succeq 0$. Please see Section \ref{mat:deg} for details. The other additional factors such as ${8}^{7\eta}$,  $d^{14\eta}$ are	artifacts of the proof technique. They primarily arise from the degree estimates of  $h_ig_i$  in Theorem \ref{comp:matPut2}. Since we want  the constant $C$  to be independent of $F$ and $G$, we retain all these factors in Theorem \ref{intro:comp:matPut3}.

	~\\
	\noindent
	{\bf Remark.}
	The assumption that $(1-\|x\|_2^2)\cdot I_{\ell}\in \qmod{G}^{\ell}$ is not serious.  When the quadratic module $\qmod{G}^{\ell}$ is Archimedean, i.e., $(r-\|x\|_2^2)\cdot I_{\ell} \in \mathrm{QM}[G]^{\ell}$ for some $r>0$,  we  have that $(1-\|x\|_2^2)\cdot I_{\ell} \in \mathrm{QM}[G(\sqrt{r}x)]^{\ell}$.   Theorem \ref{intro:comp:matPut3} remains true with $F, G$ replaced by $F(\sqrt{r}x)$, $ G(\sqrt{r}x)$. 
	~\\
	
	When $K$ is unbounded, the Putinar-type SOS representation   may not  hold. For the scalar case  (i.e.,  $F\in \mR[x]$  and $K$ is as in \reff{scalar:K}),  it was shown in \cite{putinar1999positive,putinar1999solving} that if the highest degree homogeneous term of $F$  is positive definite in $\re^n$ and $F >0$ on $K$, then there exists a degree $k$ such that 
	$(1+\|x\|_2^2)^kF $
	has a representation \reff{quamod}, 
	which is referred to as  Putinar-Vasilescu's Positivstellens{\"a}tz. A polynomial complexity  for this type of representation was proved by Mai and Magron   \cite{maim}. Recently, a matrix version of Putinar--Vasilescu's Positivstellens{\"a}tz was established in \cite{dihole}.
	
	In the following, we state a  bound similar to that in Theorem \ref{intro:comp:matPut3}  for matrix  Putinar--Vasilescu's Positivstellens{\"a}tz.
	Let $F^{hom}$  denote the highest degree homogeneous term of $F$.
	The homogenization of $F$ is denoted by $\widetilde{F}$, i.e., $\widetilde{F}(\tilde{x}):=x_0^{\deg(F)}F(x/x_0)$
	for $\tilde{x}:=(x_0,x)$.
	The notation $\lmd_{\min}(A)$ denotes the smallest eigenvalue  of the  matrix $A$.

	\begin{theorem}
		\label{int:comp:unb}
		Let $K$ be as in \reff{def:K} with $m\geq 2$.
		Suppose that $F(x)$ is an  $\ell$-by-$\ell$ symmetric polynomial matrix of degree $d$.
		If $F(x) \succ 0$ on $K$  and $F^{hom}\succ 0$ on $\mR^n$,  then for 
		\be \nonumber
		k\geq C \cdot  {8}^{7\eta}\cdot 3^{6(m-1)} (\theta(m)+2)^3(n+1)^2(\lceil d_G/2\rceil)^{6}\kappa^{7}d^{14\eta}\left(\frac{\|\widetilde{F}\|_B}{\widetilde{F}_{\min}}\right)^{7\eta+3},
		\ee
		there exists an $\ell$-by-$\ell$  SOS matrix $S$  and  polynomial matrices $P_{i}\in \mR[x]^{m\times \ell}$ $(i=1,\dots,t)$   with 
		\[
		\deg(S)\leq 2k+d,~ \deg(P_i^TGP_i)\leq 2k+d~(i=1,\dots,t), 
		\]
		such that
		\be  \label{mat:matr2}
		(1+\|x\|_2^2)^k F=S+P_1^TGP_1+\cdots + P_t^TGP_t.
		\ee
		In the above,  $\kappa$, $\eta$ are respectively the {\L}ojasiewicz constant and exponent  in \reff{def:Loj:G5}, $C>0$ is a constant independent of $F$, $G$, and  $\widetilde{F}_{\min}$
		is the optimal value of the optimization
		\be  \label{intp1}
		\left\{ \baray{rl}
		\min & \lmd_{\min}(\widetilde{F}(\tilde{x}))  \\
		\st &  	(x_0)^{2\lceil d_G/2\rceil-d_G}\widetilde{G}(\tilde{x})\succeq 0, \\
		&x_0^2 + x^Tx = 1.\\
		\earay \right.
		\ee
	\rev{Furthermore, the {\L}ojasiewicz exponent can be estimated as 
		\[
		\eta\leq (2\cdot3^{m-1}\lceil d_G/2\rceil+1)(2\cdot 3^{m}\lceil d_G/2\rceil)^{n+\theta(m)-1}.
		\]
	}	
	\end{theorem}
	
	\rev{
	For  unbounded optimization,	the positivity of $F^{\hom}$  is also crucial for obtaining SOS-type representations because  it reflects the behavior of $F$ at infinity (e.g., see \cite{hny1}). The quantity $\widetilde{F}_{\min}$ is the optimal value of the homogenized optimization (1.15). It is an important quantity for evaluating the positivity of $F$ on the unbounded set $K$, as it reflects not only the behavior of $F$ on $K$ but also its behavior at infinity. 
	Suppose that $F(x) \succeq F_{\min}\cdot I_{\ell}\succ 0$ on $K$, $F^{hom}\succeq F_{\min}^{\prime}\cdot I_{\ell}\succ 0$ on $\mR^n$, and $\tilde{x}=(x_0,x)\in \mR^{n+1}$ is a feasible point of (1.15).  One can see that $x/x_0\in K$ and  $\widetilde{F}(\tilde{x})=x_0^dF(x/x_0)\succeq x_0^d F_{\min}\cdot I_{\ell}$ for all feasible $\tilde{x}$ with $x_0\neq 0$, and $\widetilde{F}(\tilde{x})\succeq  F^{\prime}_{\min}\cdot I_{\ell}$ for all feasible $\tilde{x}$ with $x_0= 0$. 
	However, it is impossible to  bound $\widetilde{F}_{\min}$ from below with $F^{hom}$ and  $F_{\min}^{\prime}$, since $\widetilde{F}_{\min}$ can be arbitrarily close to $0$. For instance, consider the scalar polynomial $F(x)=(x-\epsilon)^2+1$ and $K=\mR$.  Note that for $x_0^2+x^2=1$, we have $\widetilde{F}(x_0,x)=(x-\epsilon \cdot x_0)^2+x_0^2=1+\epsilon^2\cdot x_0^2-2\epsilon \cdot x_0\cdot x$. Let $x_0=\sin \theta$, $x=\cos \theta$. Then,  it follows that
	\[
	\baray{ll}
	\widetilde{F}(x_0,x)=1+\epsilon^2\cdot\sin^2\theta -2\cdot \epsilon \cdot \sin \theta \cos \theta&=1+\frac{\epsilon^2}{2}\cdot (1- \cos (2\theta))-  \epsilon \cdot \sin (2\theta)\\
	&=1+\frac{\epsilon^2}{2}-\sqrt{\epsilon^2+\frac{\epsilon^4}{4}}\sin(2\theta+\phi),
	\earay
	\]
	for some $\phi\in [0,2\pi]$. This implies that $\widetilde{F}_{\min}=1+\frac{\epsilon^2}{2}-\sqrt{\epsilon^2+\frac{\epsilon^4}{4}}$. 
	One can  verify that
	$1+\frac{\epsilon^2}{2}-\sqrt{\epsilon^2+\frac{\epsilon^4}{4}}\rightarrow 0$, as $\epsilon \rightarrow \infty$, while $F_{\min}=F_{\min}^{\prime}=1$ for arbitrary $\epsilon$.
	
	}
	
	The following result  is a direct corollary of Theorem \ref{int:comp:unb}. Its proof is deferred to  Section \ref{aaaaaa}. 
\begin{corollary}
	\label{comp:unb2}
	Let $K$ be as in \reff{def:K} with $m\geq 2$.	Suppose that $F(x)$ is an  $\ell$-by-$\ell$ symmetric polynomial matrix with $\deg(F)=d$.
	If $F\succeq 0$ on $K$,  then  for any $\varepsilon>0$, we have
	\[
	(1+\|x\|_2^2)^k(F+\varepsilon \cdot (1+\|x\|_2^2)^{ \lceil \frac{d+1}{2}\rceil}I_{\ell}) \in \qmod{G}_k^{\ell},
	\]
	provided  that $
	k\geq C \cdot  \varepsilon^{-7\eta-3}.
	$
	Here,  $\eta$ is the {\L}ojasiewicz  exponent  in \reff{def:Loj:G5}, and $C>0$ is a constant that depends on   $F$, $G$.
\end{corollary}	

}

	\subsection*{Outline of the proofs}
	The proof of  Theorem \ref{intro:comp:matPut3} involves three major steps:
	
	{\it Step 1.} When $F(x)$ is an $\ell$-by-$\ell$ symmetric polynomial matrix and $K$ is described by finitely many scalar polynomial inequalities  (i.e., $K$ is as in \reff{scalar:K} ),
	we establish a polynomial bound on degrees of terms appearing in  matrix  Putinar's   Positivstellens{\"a}tz   (see Theorem \ref{comp:matPut2}). It improves the  exponential bound given by Helton and Nie \cite{helnie}.  \rev{The proof closely follows the  idea and strategies of \cite{lbbm,bmp}, adapting it to the matrix setting. This proof is provided in  Section \ref{Matrix:poly}  and is completed in three steps:}
	
	\rev{
		{\it Step 1a.}  We perturb $F$ into a polynomial matrix $P$ such that $P$ is  positive definite on the scaled simplex $\bar{\Delta}$. To be more specific, we prove that there exist $\lambda>0$ and   $h_i\in \Sigma[x]$ $(i=1,\dots,m)$  with controlled degrees   such that 
		 \[
		P(x):= F(x)-\lambda \sum_{i=1}^m h_i(x) g_i(x) I_{\ell}\succeq \frac{1}{4} F_{\min} \cdot I_{\ell},~~\forall~x\in \bar{\Delta}.
		\]
	This step uses	a result by Baldi-Mourrain-Parusinski \cite[Proposition 3.2]{bmp} about the approximation of a plateau by SOS polynomials, and some estimates on the matrix Bernstein norm (see Lemma \ref{lemma3.3}).

	{\it Step 1b.} Using the dehomogenized version of matrix  Polya's Positivstellens{\"a}tz in terms of the Bernstein basis (see Lemma \ref{mat:polya:sca}), we show that there exist matrices $P_{ \alpha}^{(k)} \succ 0$ such that
	\[
	P=\sum_{\alpha \in \mathbb{N}^n_k} P_{ \alpha}^{(k)} B_{ \alpha}^{k}(x),
	\]
	and  provide an estimate on $k$.

   {\it Step 1c.} Note that the factors of $B_{ \alpha}^{k}(x)$ satisfy $1-x_1,\dots,1-x_n,\sqrt{n}-x_1-\cdots-x_n\in \qmod{1-\|x\|_2^2}_2$. We know  that $P\in \qmod{1-\|x\|_2^2}^{\ell}$. Consequently,     a polynomial bound on the degrees of terms appearing in  matrix  Putinar's   Positivstellens{\"a}tz  can be derived from the assumption that  $ (1-\|x\|_2^2)\cdot I_{\ell}\in \qmod{G}^{\ell}$.

	}

	{\it Step 2.} When $K$ is described by a    polynomial matrix inequality (i.e., $K$ is as in \reff{def:K}), we show that $K$ can be equivalently described by finitely many scalar polynomial  inequalities with exact estimates on the quantity and degrees, by exploiting a proposition due to Schm{\"u}dgen \cite{sk09}. Additionally,  these polynomials belong to the quadratic module generated by $G(x)$ in $\mR[x]$. This is discussed in
	Section~\ref{mat:deg}.
	
	{\it Step 3.}  Finally, we prove Theorem \ref{intro:comp:matPut3}   by applying  the main result  from {\it Step 1} (i.e., Theorem \ref{comp:matPut2})  to $F$ and the equivalent scalar polynomial  inequalities for $K$ obtained in {\it Step 2}.  We refer to Section \ref{mat:mat} for details.

	The proof of Theorem \ref{int:comp:unb}  relies on Theorem \ref{intro:comp:matPut3} and  homogenization techniques. 
	
	The remaining of the paper is organized as follows. 
	Section~\ref{sc:pre}  presents some backgrounds. 
	Section~\ref{mat:mat} provides the proofs of Theorems \ref{intro:comp:matPut3}, \ref{int:comp:unb} and  studies  the  convergence rate of  matrix SOS relaxations. Conclusions and discussions are drawn  in Section \ref{sec:dis}.

	\section{Preliminaries}
	
	\label{sc:pre}
	
	\subsection{Notations}
	The symbol $\mathbb{N}$ (resp., $\mathbb{R}$) denotes the set of nonnegative integers (resp., real numbers). Let $\mathbb{R}[x]:=\mathbb{R}\left[x_1, \ldots, x_n\right]$ be the ring of polynomials in $x:=\left(x_1, \ldots, x_n\right)$ with real coefficients and let $\mathbb{R}[x]_d$ be the ring of polynomials with degrees $\leq d$.  For $\af = (\af_1, \ldots, \af_n) \in \N^n$ and a degree $d$,
	denote $|\af| \coloneqq \af_1 + \cdots + \af_n$,
	$
	\mathbb{N}_d^n:=\left\{\alpha \in \mathbb{N}^n \mid |\alpha| \leq d\right\}.
	$
	For a polynomial $p$, the symbol $\deg(p)$ denotes its total degree. For a polynomial matrix $P=(P_{ij}(x))\in \mR[x]^{\ell_1\times \ell_2}$, its degree  is  denoted by $\deg(P)$, i.e.,
	\[
	\deg(P):=\max\{\deg(P_{ij}(x)),~1\leq i\leq \ell_1,1\leq j\leq \ell_2\}.
	\] 
	For a real number $t,\lceil t\rceil$ denotes the smallest integer  greater than or equal to $t$. For a matrix $A$, $A^{T}$ denotes its transpose and $\|A\|_2$ denotes the spectral norm of  $A$.  
	A symmetric matrix $X \succeq 0$ (resp., $X \succ 0$)  if $X$ is positive semidefinite (resp., positive definite).  
	For a smooth function $p$ in $x$, its gradient  is denoted by $\nabla p$.

	
	\subsection{Some basics in real algebraic geometry} \label{sec:qm}
	
	In this subsection, we review some basics in real algebraic geometry. We refer to \cite{cj,hdlb,kism,kojm,Lau09,niebook,schhol,sk09} for more details. 
	
	A polynomial matrix $S$ is said to be a sum of squares (SOS) if $S=P^TP$ for some polynomial matrix $P(x)$.
	The set of all $\ell$-by-$\ell$  SOS polynomial matrices in $x$ is denoted as $\Sigma[x]^{\ell}$. 
	For a  degree $k$, denote the truncation
	\[
	\Sigma[x]_{k}^{\ell} \, \coloneqq  \,  \left\{
	\sum_{i=1}^{t} P_{i}^T P_{i}   \mid
	t \in \mathbb{N}, ~ P_{i} \in \mathbb{R}[x]^{\ell \times \ell},~ 2\deg(P_{i})\leq k\right\}.
	\]
	For an $m$-by-$m$   symmetric  polynomial matrix $G$, the  quadratic module of $\ell$-by-$\ell$  polynomial matrices generated
	by  $G(x)$ in $\mR[x]^{\ell\times \ell}$ is
	\be \nonumber 
	\qmod{G}^{\ell}=\left\{
	S+\sum_{i=1}^t P_i^TGP_i \mid t\in \mathbb{N},~ S \in \Sigma[x]^{\ell},~P_i \in \mathbb{R}[x]^{m\times \ell}
	\right\}.
	\ee
	The $k$th degree truncation of $\qmod{G}^{\ell}$ is
	\be
	\qmod{G}_{k}^{\ell} = \left\{
	S+\sum_{i=1}^t P_i^TGP_i \mid t\in \mathbb{N},~ S \in \Sigma[x]^{\ell}_k,~P_i \in \mathbb{R}[x]_{\lceil k/2- \deg(G)/2\rceil }^{m\times \ell}
	\right\}.
	\ee
	Note that $\Sigma[x]^{\ell} \subseteq \qmod{G}^{\ell}$.
	For the case $\ell=1$, we  simply denote 
	\[
	\Sigma[x]_k:=\Sigma[x]_k^{1},~\rev{~\qmod{G}:=\qmod{G}^{1}},~\qmod{G}_k:=\qmod{G}_k^{1}.
	\]

	The  quadratic module $\mathrm{QM}[G]^{\ell}$ is said to be Archimedean if  there exists  $r>0$ such that
	$
	\left(r-\|x\|_2^2\right) \cdot I_{\ell} \in \mathrm{QM}[G]^{\ell} .
	$
	The polynomial matrix $G$ determines the semialgebraic set
	$$
	K:=\left\{x \in \mathbb{R}^n \mid G(x) \succeq 0 \right\} .
	$$
	If $\qmod{G}^{\ell}$ is Archimedean, then the set $K$ must be compact. However, when $K$ is compact, \rev{it is not necessarily that $\qmod{G}^{\ell}$ is Archimedean (see \cite{JPre}).}  The following is   known as the matrix-version of Putinar's Positivstellens{\"a}tz.
	
	\begin{theorem}[\cite{kojm,schhol}]	
		Suppose that $G$ is an $m$-by-$m$  symmetric polynomial matrix  and $\qmod{G}^{\ell}$ is Archimedean. If the $\ell$-by-$\ell$  symmetric polynomial matrix $F\succ 0$  on $K$, then we have $F \in \qmod{G}^{\ell}$.
	\end{theorem}

	\subsection{Some technical propositions} \label{sec:qm1}
	In this subsection, we present several useful propositions that are used in our proofs.
	The following is the classical {\L}ojasiewicz inequality. 
	
	\begin{corollary}[\cite{bcr,LojS}]
		\label{loj}
		Let $A$ be a closed and bounded semialgebraic set, and  let $f$ and  $g$  two continuous semialgebraic functions from $A$ to $\mR$ such that $f^{-1}(0) \subseteq g^{-1}(0)$. Then there exist two positive constants $\eta$, $\kappa$ such that:
		$$
		|g(x)|^{\eta} \leq \kappa |f(x)|,~\forall x \in A .
		$$
	\end{corollary}
	\noindent
	The smallest constant $\eta$  is called the {\L}ojasiewicz exponent and $\kappa$ is the corresponding {\L}ojasiewicz
	constant.

	The Markov inequality below can be used to estimate gradients of polynomials.
	\begin{theorem} [\cite{bmp,aksr}]
		\label{Markov}
		Let $\bar{\Delta}$ be the scaled simplex as in \reff{sca:simplex} and let $p \in \mathbb{R}[x]$ be a polynomial with degree $\leq d$. Then:
		\[
		\max _{x \in \bar{\Delta}}\|\nabla p(x)\|_2 \leq \frac{2 d(2 d-1)}{\sqrt{n}+1}\max _{x \in \bar{\Delta}}|p(x)|.
		\]
	\end{theorem}

\rev{	
The following theorem is a slight variation of \cite[Theorem 3.3]{hapham}, which gives explicit {\L}ojasiewicz exponents for basic closed semialgebraic sets.
\begin{theorem}[\cite{hapham}] \label{klpham}
Suppose that $g_1, \dots, g_s, h_1,\dots, h_l\in \mR[x]_d$ and let
$$
	S:=\{x \in \mathbb{R}^n \mid h_1(x)=\cdots =h_l(x)=0, 	g_1(x) \geq 0, \ldots, g_s(x) \geq 0\} \neq \emptyset.
$$
Then for any compact set $K \subset \mathbb{R}^n$, there is a constant $c>0$ such that for all $x \in K$,	
\[
 d(x, S)^{\mathcal{R}(n+l+s-1, d+1)}  \leq  -c\cdot \min\{g_1(x),\dots,g_s(x),h_1(x),-h_1(x),\dots,h_l(x),-h_l(x),0\},
\]
where $d(x,S)$ is the Euclidean distance to the set $S$ and $ \mathcal{R}(n, d)=d\cdot(3 d-3)^{n-1} $.
\end{theorem} 
}

	\section{Degree bounds for  scalar polynomial inequalities}
	\label{Matrix:poly}
\rev{	In this section, we consider the case where  $ G(x)$ is  a diagonal matrix with diagonal entries  $g_1,\dots,g_m \in \mR[x]$. Then, the set $K$ as in \reff{def:K} can be equivalently described as
	\be \label{scal:K}
	K=\{x\in\mR^n \mid g_1(x)\geq 0,\dots,\dots,g_m(x)\geq 0\}.
	\ee
	If the symmetric    polynomial matrix $F(x)\succ 0$ on $K$, we prove a  polynomial bound on the degrees of terms appearing in  matrix  Putinar's Positivstellens{\"a}tz. This  improves the  exponential bound  shown by Helton and Nie \cite{helnie}. The proof essentially generalizes  the idea and strategies used in \cite{lbbm,bmp} to the matrix case. We refer {\it to Step 1 in the section `Outlines of proofs' for an overview of the proof.} }
	
	Let $\Delta^n$ be the $n$-dimensional standard simplex, i.e.,
	\[
	\Delta^n=\{x\in\mR^n\mid x_1+\dots+x_n=1,x_1\geq 0,\dots,x_n\geq 0\}.
	\] 
	Denote the scaled simplex by
	\[
	\bar{\Delta}=\left\{x \in \mathbb{R}^n \mid \sqrt{n}-x_1-\cdots-x_n \geq 0,1+x_1 \geq 0, \ldots, 1+x_n \geq 0\right\}.
	\]
	For a degree $t$,   let  $\{B_{ \alpha}^{(t)}(x):\alpha \in \mathbb{N}^n_t \}$ be the Bernstein basis of degree $t$ on $\bar{\Delta}$, i.e.,
	{\small
		\be \label{bbasis}
		B_{ \alpha}^{(t)}(x)=\frac{t!}{\alpha_1!\cdots\alpha_n!}(n+\sqrt{n})^{-t}\left(\sqrt{n}-x_1-\cdots-x_n\right)^{t-|\alpha|}\left(1+x_1\right)^{\alpha_1} \cdots\left(1+x_n\right)^{\alpha_n},
		\ee
	}
	For a degree $t\geq d$, the $\ell\times \ell$ symmetric polynomial matrix $F$ of degree $d$ can be expressed in the  Bernstein basis $
	\{B_{ \alpha}^{(t)}(x):\alpha \in \mathbb{N}^n_t \}$   as 
	\[
	F=\sum_{\alpha \in \mathbb{N}_t^n} F_{ \alpha}^{(t)} B_{ \alpha}^{(t)}(x),
	\]
	where $F_{ \alpha}^{(t)}$ are $\ell$-by-$\ell$ symmetric matrices.
	The Bernstein norm  of $ F$ with respect to $
	\{B_{ \alpha}^{(t)}(x):\alpha \in \mathbb{N}^n_t \}$  is defined as 
	\[
	\|F\|_{B,t}=\max\{\|F_{\alpha}^{(t)}\|_2:~\alpha \in \mathbb{N}^n_{t}\}.
	\]
	For convenience, we  denote $\|F\|_{B,d}$ simply as $\|F\|_{B}$. Similarly, 
	$F$ can also be expressed in the monomial basis $\{x^{\alpha}:\alpha \in \mathbb{N}^n_d\}$ as 
	\[
	F=\sum_{\alpha \in \mathbb{N}_d^n}F_{\alpha}\frac{d!}{\alpha_1!\cdots\alpha_n!}x^{\alpha}.
	\] 
	The spectral norm of $F$ is defined as 
	$
	\|F\|_2=\max\{\|F_{\alpha}\|_2:~\alpha \in \mathbb{N}^n_{d}\}.
	$
	\rev{
	For $f(x)\in \mR[x]_{d_1}$, $g(x) \in \mR[x]_{d_2}$ and $\ell_1\geq \ell_2\geq d_1$, the following is a basic property of the Bernstein norm  (see \cite{gefa}, \cite[Lemma 1.5]{bmp}):
	\be \label{proB}
	\max\limits_{s \in \bar{\Delta}}|f(s)|\leq \|f\|_{B,\ell_1}\leq \|f\|_{B,\ell_2}, ~~ 
 \|fg\|_{B,d_1+d_2}\leq \|f\|_{B,d_1}\|g\|_{B,d_2}. 
	\ee
	
	}

	Suppose that $p(x)$ is a homogeneous polynomial   with $\deg(p)=d$.  Polya's Positivstellens{\"a}tz \cite{polya,pvrb} states that if $p(x)\geq p_{\min}>0$ on the standard simplex $\Delta^n$, then the polynomial $(x_1+\cdots+x_n)^kp(x)$  has only positive coefficients if
	$
	k> d(d-1)\frac{\|p\|}{2p_{\min}}-d.
	$ 
	This result has been generalized to the matrix case by Scherer and Hol \cite{schhol}.  It was shown in \cite[Theorem 3]{schhol} that 	if $F(x)$ is an $\ell$-by-$\ell$   homogeneous  symmetric polynomial matrix with $\deg(F)=d$ and $F(x)\succeq F_{\min} \cdot  \rev{I_{\ell}}\succ 0$ on  $\triangle^n$, then for
	$
	k> d(d-1)\frac{\|F\|_2}{2F_{\min}}-d,
	$
	all coefficient matrices  of $(x_1+\cdots+x_n)^k F(x)$ in the monomial  basis are positive definite.

	In the following, we present a dehomogenized version of matrix  Polya's Positivstellens{\"a}tz in terms of the Bernstein basis.
	\begin{lemma}
		\label{mat:polya:sca}
		Suppose  $F(x)$ is an $\ell$-by-$\ell$  symmetric    polynomial matrix with $\deg(F)=d$.
		If  $F(x)\succeq F_{\min} \cdot  \rev{I_{\ell}}\succ 0$ on $\bar{\Delta}$, then for 
		\[
		k> d(d-1)\frac{\|F\|_B}{2F_{\min}}-d,
		\]
		all coefficient matrices  of $ F(x)$ in the Bernstein  basis $\{ B_{ \alpha}^{(k+d)}(x):\alpha \in \mathbb{N}^n_{k+d}\}$ are positive definite.
	\end{lemma}
	
	\begin{proof}
		Suppose that	  $F$  can be expressed 
		in the  Bernstein basis   $
		\{ B_{ \alpha}^{(d)}(x):\alpha \in \mathbb{N}^n_d\} $ as 
		\[
		F=\sum_{\alpha \in \mathbb{N}^n_d}  F_{ \alpha}^{(d)}  B_{ \alpha}^{(d)}(x),
		\]
		where $F_{ \alpha}^{(d)}$ are $\ell$-by-$\ell$ symmetric matrices.
		Denote
		\[
		\widehat{F}(y_1,\dots,y_{n+1})=\sum_{\alpha \in \mathbb{N}^n_d} F_{ \alpha}^{(d)} \frac{d!}{\alpha_1!\cdots\alpha_n!} y^{\alpha}y_{n+1}^{d-|\alpha|}.
		\] 
		For each $(y_1,\dots,y_{n+1}) \in \Delta^{n+1}$, let
		\[
		x_i=(n+\sqrt{n})y_i-1~(i=1,\dots,n).
		\]
		Then we know that 
		\[
		\baray{ll}
		\sqrt{n}-\sum\limits_{i=1}^n x_i=\sqrt{n}-\sum\limits_{i=1}^n ((n+\sqrt{n})y_i-1)
		&=\sqrt{n}+n-(n+\sqrt{n})\sum\limits_{i=1}^n y_i\\
		&=(\sqrt{n}+n)(1-\sum\limits_{i=1}^n y_i)\\
		&\geq 0,\\
		\earay
		\]
		which implies that $x=(x_1,\dots,x_n)\in \bar{\Delta}$.
		By computations, we have that for $(y_1,\dots,y_{n+1}) \in \Delta^{n+1}$, the following holds
		\[
		\widehat{F}(y_1,\dots,y_{n+1})=\sum_{\alpha \in \mathbb{N}^n_d} F_{ \alpha}^{(d)} B_{ \alpha}^{(d)}(x)=F(x)\succeq F_{\min}\cdot I_{\ell}.
		\]
		Hence, $\widehat{F} \succeq F_{\min}\cdot I_{\ell}$ on 
		$\Delta^{n+1}$.
		It follows from  \cite[Theorem 3]{schhol}  that for $k> d(d-1)\frac{\|\widehat{F}\|_2}{2F_{\min}}-d$, all coefficient matrices  of $(y_1+\cdots+y_{n+1})^k\widehat{F}(y_1,\dots,y_{n+1})$ in the monomial  basis are positive definite, i.e., there exist $P_\alpha \succ 0$ such that 
		{\small
			\[
			(y_1+\cdots+y_{n+1})^k\widehat{F}(y_1,\dots,y_{n+1})=\sum_{\alpha \in \mathbb{N}_{k+d}^{n}} P_\alpha \frac{(k+d)!}{\alpha_1!\cdots\alpha_n!} y_1^{\alpha_1} \cdots y_n^{\alpha_n} (y_{n+1})^{k+d- |\alpha |}.
			\]
		}
		By substituting 
		\[
		y_i=\frac{x_i+1}{n+\sqrt{n}}~(i=1,\dots,n),\,y_{n+1}=1- \frac{1}{n+\sqrt{n}}\left(\sum\limits_{i=1}^nx_i+n\right)
		\]
		into the above identity, we obtain 
		\[
		F(x_1,\dots,x_{n})=\sum_{\alpha \in \mathbb{N}_{k+d}^{n}} P_\alpha  B_{ \alpha}^{(k+d)}(x).
		\]
		Since $\|\widehat{F}\|_2= \|F\|_{B}$, the conclusion follows.
		
		$\qed$
	\end{proof}

	\begin{lemma}
		\label{lemma3.3}
		Suppose  $P(x)$ is an $\ell$-by-$\ell$ symmetric polynomial matrix  with $\deg(P)=d_P$. Let 	$\xi\in \mR^{\ell}$ with $\|\xi\|_2=1$, and denote 
		\[
		P^{(\xi)}(x):=\xi^T P(x) \xi.
		\]
		Then, we have that $\| P^{(\xi)}\|_{B,d_P}\leq \|P\|_B$.
		Furthermore, there exists $\xi_0 \in \mR^{\ell}$ with $\|\xi_0\|_2=1$ such that $\| P^{(\xi_0)}\|_{B,d_P}=\|P\|_B$.
	\end{lemma}
	
	\begin{proof}
		Suppose	that $P$  can be expressed  in the  Bernstein basis   $
		\{ B_{ \alpha}^{(d_P)}(x):\alpha \in \mathbb{N}^n_{d_P}\} $ as 
		\[
		P=\sum_{ \alpha \in \mathbb{N}^n_{d_P}}P_{ \alpha}^{(d_P)} B_{ \alpha}^{(d_P)}(x),
		\]
		where $P_{ \alpha}^{(d_P)}$ are $\ell$-by-$\ell$ symmetric matrices.
		Note that
		\be \nonumber
		\baray{ll}
		\| P^{(\xi)}\|_{B,d_P}=\| \xi^T P(x) \xi\|_{B,d_P}&= \|\sum\limits_{\alpha \in \mathbb{N}^n_{d_P}} \xi^TP_{ \alpha}^{(d_P)}\xi B_{ \alpha}^{(d_P)}(x)\|_{B,d_P}\\
		&=\max\{|\xi^TP_{ \alpha}^{(d_P)}\xi|:~\alpha \in \mathbb{N}^n_{d_P}\}\\
		&\leq \max  \{\|P_{ \alpha}^{(d_P)}\|_2:~\alpha \in \mathbb{N}^n_{d_P}\} = \|P\|_B.
		\earay
		\ee
		Thus, we have 
		$
		\| P^{(\xi)}\|_{B,d_P}\leq  \|P\|_B.
		$  
		Since there exists $\xi_0 \in \mR^{\ell}$ with $\|\xi_0\|_2=1$ such that  	
		\[
		\|P\|_B=\max\{|\xi_0^TP_{\alpha}^{(d_P)}\xi_0 |:~\alpha \in \mathbb{N}^n_{d_P}\},
		\]
		then we have that
		\[
		\| P^{(\xi_0)}\|_{B,d_P}=\max\{|\xi_0^TP_{ \alpha}^{(d_P)}\xi_0|:~\alpha \in \mathbb{N}^n_{d_P}\}=\|P\|_B.
		\]
		
		$\qed$
	\end{proof}

	Denote
	\be \label{def:Loj:G3}
	\mathcal{G}(x)=-\min \left(\rev{\frac{g_1(x)}{\left\|g_1\right\|_B}}, \ldots, \rev{\frac{g_{m}(x)}{\left\|g_{m}\right\|_B}}, 0\right). 
	\ee
	For $x\in \mR^n$,  let $d(x,K)$ denote the Euclidean distance to the set $K$, i.e.,
	\[
	d(x,K)=\min\{\|x-y\|_2,\,\, y\in K\}.
	\]
	If $\mathcal{G}(x)=0$, then we have $x\in K$ and $ d(x,K)=0$. \rev{The distance function $d(x,K)$ is a continuous semialgebraic function (see \cite[Proposition 2.2.8]{bcr}, \cite[Proposition 1.8]{hapham}).}
	By Corollary \ref{loj}, there exist positive constants $\eta$, $\kappa$ such that the following {\L}ojasiewicz inequality holds:
	\be \label{Lojexp2}
	d(x,K)^{\eta} \leq \kappa \mathcal{G}(x), \quad \forall x \in \bar{\Delta}.
	\ee
	Note that
	\[
	d_G=\max\{\deg(g_1),\dots,\deg(g_m)\}.
	\]
	
	In the following, we prove a polynomial bound on the degrees of terms appearing in matrix   Putinar's Positivstellens{\"a}tz  when $K$ is given by finitely many scalar polynomial inequalities, which  improves the  exponential bound established in \cite{helnie}. 
	\rev{Recall that the quadratic module of scalar  polynomials generated 
		by  $G(x)$ in $\mR[x]$ is denoted by $\qmod{G}$ (see Section \ref{sec:qm}).
	}
	
	\begin{theorem}
		\label{comp:matPut2}
		Let $ G(x)$ be  a diagonal matrix with diagonal entries  $g_1,\dots,g_m \in \mR[x]$ and let $K$ be as in \reff{scal:K}.	Suppose that $F(x)$ is an $\ell$-by-$\ell$ symmetric polynomial matrix with $\deg(F)=d$, and \rev{$ (1-\|x\|_2^2)\cdot I_{\ell} \in\qmod{G}^{\ell}$.}
		If $F(x)\succeq F_{\min} \cdot  I_{\ell} \succ 0$ on $K$,  then $F \in \qmod{G}_k^{\ell}$ for
		\be \label{bound1}
		k\geq C\cdot {8}^{7\eta} m^3n^2d_{G}^6\kappa^{7}d^{14\eta}\left(\frac{\|F\|_B}{F_{\min}}\right)^{7\eta+3}.
		\ee
		In the above,  $\kappa$, $\eta$ are respectively the {\L}ojasiewicz constant and exponent  in \reff{Lojexp2} and $C>0$ is a constant independent of $F$, $g_1,\dots,g_m$. \rev{Furthermore, the {\L}ojasiewicz exponent can be estimated as 
		$	\eta\leq (d_G+1)(3d_G)^{n+m-2}$.
		}
	\end{theorem}
	
		~\\
	\noindent
	{\bf Remark.} 
	\rev{At a feasible point $x\in K$, the Mangasarian--Fromovitz constraint
		qualification (MFCQ) holds if there exists a direction $v\in \mR^n$ such that $\nabla g_j(x)^Tv>0$	 for all active inequality constraints $g_j(x)=0$. The  linear independence constraint qualification (LICQ)  holds at $x$ if the set of gradient vectors $\nabla g_j(x)$ of all active constraints  is	linearly independent.
		It was shown in \cite{Robinson} that  if the MFCQ holds at every $x\in K$, then the {\L}ojasiewicz exponent, as in \reff{Lojexp2}, is equal to 1. Since  the MFCQ is weaker than the LICQ,
		we know that if the LICQ holds at every $x\in K$,  the {\L}ojasiewicz exponent $\eta=1$ (see also \cite{lbbm,bmp})\footnote{The author would like to thank the referee for pointing out this fact.}.
		Then, under the assumptions of Theorem   \ref{comp:matPut2}, we have that   $\rev{F} \in \qmod{G}_k^{\ell}$ if
		\be \label{rate10}
		k\geq C\cdot  m^3n^2d_{G}^6\kappa^{7}d^{14}\left(\frac{\|F\|_B}{F_{\min}}\right)^{10}.
		\ee
	}

	\begin{proof}
		Without loss of generality,   we assume that $\left\|g_i\right\|_B=1$ for $i=1,\dots,m$. Otherwise, we can 
		scale $g_i$ by $1/\|g_i\|_B$. \rev{We remark that scaling $g_i$ does not change $\kappa$, since we normalize $g_i$ in $\mathcal{G}$ and $(g_i(x)/\|g_i\|_B)/(\|g_i(x)/\|g_i\|_B\|_B)=g_i(x)/\|g_i\|_B$. Thus,   the bound remains unchanged.} Let 
		\[
		\delta=\frac{1}{{8}^{\eta}\kappa d^{2\eta}}\left(\frac{F_{\min}}{\|F\|_B}\right)^{\eta},~
		v =    \frac{\delta F_{\min}}{20m\|F\|_B},~ \lambda=5 \|F\|_B/\delta.
		\]
		For the above $\delta$, $v$, it follows from  \cite[Proposition 3.2]{bmp}  that there exists $h_i\in \mR[x]$ such that for $x\in \bar{\Delta}$, we have that
		\bit 
		\item[(i)] if $g_i(x) \geq 0,\left|h_{i}(x)\right| \leq 2 v$.
		\item[(ii)]  if $g_i(x) \leq-\delta,\left|h_{i}(x)\right| \geq \frac{1}{2}$.
		\item[(iii)]  $\left\|h_{i}\right\|_B \leq 1$.
		\item[(iv)] $h_{i} \in \Sigma[x]_{k}$ with $k\leq  10^5n d_{G}^2 \delta^{-2} v^{-1}$.
		\eit 
		Let
		\[
		P= F-\lambda \sum_{i=1}^m h_i g_i I_{\ell}.
		\]
		In the following, we show that
		$P\succeq \frac{1}{4} F_{\min} \cdot I_{\ell}
		$ on $\bar{\Delta}$. 
		\rev{Equivalently, for a fixed $x\in \bar{\Delta}$, we prove that  for each  $\xi\in \mR^{\ell}$ with $\|\xi\|_2=1$, the following holds:}
		\be \label{pos2}
		F^{(\xi)}(x)-\lambda \sum_{i=1}^m h_i(x) g_i(x) =\xi^T(F(x)-\lambda \sum_{i=1}^m h_i(x) g_i(x) I_{\ell})\xi \geq \frac{1}{4} F_{\min}.
		\ee
		In the above, $
		F^{(\xi)}(x):=\xi^T F(x) \xi . $
		
	\rev{	
		{\it Case I.}	Suppose that  $ F^{(\xi)}(x) \leq \frac{3F_{\min}}{4}$. Then, we know that $x\notin K$. Let  $x^*\in K$ be a vector such that $\|x-x^*\|_2=d(x,K)$. Since $F(x^*)\succeq F_{\min}\cdot I_{\ell}$, we have that 
		\[
		\baray{ll}
		\frac{F_{\min}}{4} \leq F^{(\xi)}(x^*)-F^{(\xi)}(x) &\leq \max\limits _{s \in \bar{\Delta}}\|\nabla F^{(\xi)}(s)\|_2\|x-x^*\|_2\\
		&  \leq  \frac{2d(2d-1)}{\sqrt{n}+1}    \|x-x^*\|_2 \max\limits _{s \in \bar{\Delta}}|F^{(\xi)}(s)|,    \\
		\earay
		\]
		where the last inequality follows from Theorem \ref{Markov}.   Since $\max\limits_{s \in \bar{\Delta}}|F^{(\xi)}(s)|\leq \|F^{(\xi)}\|_{B,d}$ (see \reff{proB}), it holds that
		\be \label{eq00}
		F_{\min}\leq \max _{s \in \bar{\Delta}}|F^{(\xi)}(s)|\leq \|F^{(\xi)}\|_{B,d}\leq \|F\|_B,
		\ee
		where the last inequality is due to Lemma \ref{lemma3.3}.
		Then, we know that
		\[
		\frac{F_{\min}}{4} \leq \frac{2d(2d-1)}{\sqrt{n}+1}\|F^{(\xi)}\|_{B,d}   \|x-x^*\|_2\leq 2d^2\|F\|_B\|x-x^*\|_2.
		\]
		By the {\L}ojasiewicz inequality \reff{Lojexp2}, we have that 
		\[
		\mathcal{G}(x)\geq \kappa^{-1}d(x,K)^{\eta}\geq \kappa^{-1}\|x-x^*\|_2^{\eta}\geq \frac{1}{{8}^{\eta}\kappa d^{2\eta}}(\frac{F_{\min}}{\|F\|_B})^{\eta}=\delta.
		\]
		From the definition of the function $\mathcal{G}(x)$ as in \reff{def:Loj:G3}, there exists $i_0\in \{1,\dots,m\}$  such that $g_{i_0}(x) \leq-\delta$. Since $h_{i_0}(x)$ is an SOS (see item (iv)) and $g_{i_0}(x) \leq-\delta$, we have $h_{i_0}(x) \geq \frac{1}{2}$ by item (ii). 
		Note that
		\[
		-h_{i_0}(x)g_{i_0}(x)\geq \frac{1}{2} \delta, ~ -h_{i}(x)g_{i}(x) \geq -2v\cdot \max\limits_{s \in \bar{\Delta}} |g_i(s)|\geq  -2v\cdot \|g_i\|_B= -2v~(i=1,\dots,m).
		\]
	In the above, we use the inequality \reff{proB}.	Then, we have that
		$$
		\begin{aligned}
			F^{(\xi)}(x) -\lambda \sum_{i=1}^m h_i(x) g_i(x)& \geq F^{(\xi)}(x)+\lambda \cdot\frac{1}{2} \delta-2 \lambda v(m-1) \\
			& \geq F^{(\xi)}(x)+\lambda \frac{\delta}{4}+\lambda\left(\frac{\delta}{4}-2 v(m-1)\right). \\
		\end{aligned}
		$$
		One can see  from the relation \reff{eq00} that $F^{(\xi)}(x)\geq -\|F^{(\xi)}\|_{B,d}$ and $ F_{\min}\leq \|F\|_B$, which implies that
		\[
		\frac{\delta}{4}-2 v(m-1)\geq 	\frac{\delta}{4}-2 vm\geq \frac{\delta}{4}- \frac{\delta F_{\min}}{10\|F\|_B}  \geq \frac{\delta}{4}-\frac{\delta}{10}>0,
		\]	
		\[
		\lambda \frac{\delta}{4}= \frac{5}{4}\|F\|_{B}\geq \frac{5}{4}\|F^{(\xi)}\|_{B,d}.
		\]
		Then, it follows that
		$$
		\begin{aligned}
			F^{(\xi)}(x) -\lambda \sum_{i=1}^m h_i(x) g_i(x)
			& \geq F^{(\xi)}(x)+\lambda \frac{\delta}{4}+\lambda\left(\frac{\delta}{4}-2 v(m-1)\right) \\
			&\geq -\|F^{(\xi)}\|_{B,d}+\frac{5}{4}\|F^{(\xi)}\|_{B,d}\\
			&\geq \frac{1}{4} \|F^{(\xi)}\|_{B,d} \geq \frac{1}{4} \max _{s \in \bar{\Delta}}|F^{(\xi)}(s)|\\
			&\geq \frac{1}{4} F_{\min},
		\end{aligned}
		$$
		where the second inequality follows from the inequality \reff{proB}.
}

\rev{	
		{\it Case II.}  Suppose that  $ F^{(\xi)}(x) \geq \frac{3F_{\min}}{4}$.  Note that $-h_{i}(x)g_{i}(x) \geq -2v\cdot \|g_i\|_B= -2v$ for $i=1,\dots,m$, and
		\[
		2 m \lambda v=2m\cdot \frac{\delta F_{\min}}{20m\|F\|_B}\cdot 5 \|F\|_B/\delta=\frac{1}{2}F_{\min}.
		\]
		Then, we have that
		\[
		F^{(\xi)}(x) -\lambda \sum_{i=1}^m h_i(x) g_i(x)\geq \frac{3}{4} F_{\min}-2 m \lambda v\geq \frac{1}{4}F_{\min}.
		\]
	This concludes the proof of \reff{pos2}.
		
		Since  $\xi\in \mR^n$ with $\|\xi\|_2=1$ is arbitrary, we have  $P(x)\succeq \frac{1}{4} F_{\min}\cdot I_{\ell}
		$ on $\bar{\Delta}$.
		Note that $x$ can be any fixed point in $\bar{\Delta}$,  we know that $P(x)\succeq \frac{1}{4} F_{\min}\cdot I_{\ell}$ for all $x \in \bar{\Delta}$.
}
		\rev{		Since $\deg(h_i)\leq 10^5n d_{G}^2 \delta^{-2} v^{-1}$ for $i=1,\dots,m$,  we have that
		\[
		\baray{ll}
		\deg(h_ig_i)\leq  10^5 n d_{G}^3 \delta^{-2} v^{-1} &\leq 2\cdot 10^{6}  mn d_{G}^3 \delta^{-3}    \frac{\|F\|_B}{F_{\min}}\\
		& \leq 2\cdot 10^{6} \cdot  {8}^{3\eta} mnd_{G}^3\kappa^{3}d^{6\eta}(\frac{\|F\|_B}{F_{\min}})^{3\eta+1}.     \\	
		\earay
		\]
		By Lemma \ref{lemma3.3}, there exists $\xi_0\in \mR^n$ with $\|\xi_0\|_2=1$ such that $\|P\|_B=\|\xi_0^TP\xi_0\|_{B,\deg(P)}$. Then, we obtain that
		\[
		\baray{ll}
		\|P\|_B=\|\xi_0^TP\xi_0\|_{B,\deg(P)}
		&=
		\|F^{(\xi_0)}-\lambda \sum_{i=1}^m h_i g_i\|_{B,\deg(P)}\\
		&\leq \|F^{(\xi_0)}\|_{B,\deg(P)}+\lambda \sum_{i=1}^m\|h_ig_i\|_{B,\deg(P)}\\
		&\leq \|F^{(\xi_0)}\|_{B,\deg(F)}+\lambda \sum_{i=1}^m\|h_i\|_{B,\deg(P)-\deg(g_i)}\|g_i\|_{B,\deg(g_i)} \\
		&\leq \|F\|_{B}+\lambda m\leq (1+5\delta^{-1}m)\|F\|_B\\
		&\leq   {8}^{\eta+1}m\kappa d^{2\eta}(\frac{\|F\|_B}{F_{\min}})^{\eta}\|F\|_B,\\
		\earay
		\]
	where the second inequality is due to the inequality \reff{proB}.	Note that
		\[
		\baray{ll}
		\deg(P)=\deg(F-\lambda \sum_{i=1}^m h_i g_i I_{\ell})&=\max\{\deg(F),\deg(h_i g_i)~(i=1,\dots,m)\}\\
		&\leq \max\{d, 2\cdot 10^{6} \cdot  {8}^{3\eta} mnd_{G}^3\kappa^{3}d^{6\eta}(\frac{\|F\|_B}{F_{\min}})^{3\eta+1}\}\\
		& =2\cdot 10^{6} \cdot  {8}^{3\eta} mnd_{G}^3\kappa^{3}d^{6\eta}(\frac{\|F\|_B}{F_{\min}})^{3\eta+1}.
		\earay
		\]
			It follows from  Lemma \ref{mat:polya:sca} that all coefficient matrices  of $ P(x)$ in the Bernstein  basis $\{ B_{ \alpha}^{(k)}(x):\alpha \in \mathbb{N}_{k}^n\}$ are positive definite, for
		\[
		\baray{ll}
		k&\geq (16\cdot 10^{6})^2 \cdot {8}^{7\eta} m^3n^2d_{G}^6\kappa^{7}d^{14\eta}(\frac{\|F\|_B}{F_{\min}})^{7\eta+3}   \\
		&>  \left(2\cdot 10^{6} \cdot  {8}^{3\eta} mnd_{G}^3\kappa^{3}d^{6\eta}(\frac{\|F\|_B}{F_{\min}})^{3\eta+1}\right)^2 {8}^{\eta+2}m\kappa d^{2\eta}(\frac{\|F\|_B}{F_{\min}})^{\eta}\frac{\|F\|_B}{F_{\min}}\\
		&\geq \deg(P)^2\frac{\|P\|_B}{\frac{1}{2}F_{\min}},\\
		\earay
		\]
		 i.e.,  there exist matrices $P_{ \alpha}^{(k)} \succ 0$ such that
		\[
		F-\lambda \sum_{i=1}^m h_i g_i I_{\ell}=\sum_{\alpha \in \mathbb{N}_{k}^n} P_{ \alpha}^{(k)} B_{ \alpha}^{k}(x).
		\]

		In the following, we show that
		$\sqrt{n}-x_1-\cdots-x_n\in \qmod{1-\|x\|_2^2}_2$. Consider the optimization 
		\be  \label{nlplo}
		\left\{ \baray{rl}
		\min & \sqrt{n}-x_1-\cdots-x_n  \\
		\st & 1-\|x\|_2^2\geq 0. \\
		\earay \right.
		\ee
		One can verify that the optimal value of 
		\reff{nlplo} is 0. Since $\sqrt{n}-x_1-\cdots-x_n$ and $-(1-\|x\|^2)$ are SOS convex and the feasible set has an interior point, it follows from \cite[Theorem 3.3]{Lasconvex} that the lowest order SOS relaxation of \reff{nlplo} is exact and the optimal values of all SOS relaxations are achievable. Consequently, we have that $\sqrt{n}-x_1-\cdots-x_n\in \qmod{1-\|x\|_2^2}_2$. 
		Similarly, we know that $1+x_1,\cdots,1+x_n\in \qmod{1-\|x\|_2^2}_2$.
		Hence, we have
		{\small
			\[
			\left(\sqrt{n}-x_1-\cdots-x_n\right)^{k-|\alpha|}\left(1+x_1\right)^{\alpha_1} \cdots\left(1+x_n\right)^{\alpha_n} \in \qmod{1-\|x\|_2^2}_{2(k-|\alpha|)+2\alpha_1+\dots+2\alpha_n}=\qmod{1-\|x\|_2^2}_{2k}.
			\]
		}
		Then, from the definition of  $B_{ \alpha}^{k}(x)$, we have  $B_{ \alpha}^{k}(x)\in \qmod{1-\|x\|_2^2}_{2k}$.
		Since  $(1-\|x\|_2^2)\cdot I_{\ell} \in\qmod{G}_{k_0}^{\ell} $ for some $k_0>0$, it follows that  $1-\|x\|_2^2 \in\qmod{G}_{k_0}$, and thus $B_{ \alpha}^{k}(x)\in \qmod{1-\|x\|_2^2}_{2k+k_0}$. Then, the conclusion holds for sufficiently large $C$. Note that
		\[
			K=\left\{x\in\mR^n \mid \frac{g_1(x)}{\left\|g_1\right\|_B}\geq 0,\dots,\frac{g_m(x)}{\left\|g_m\right\|_B}\geq 0\right\}.
			\]
			It follows from Theorem \ref{klpham} that 	$\eta\leq (d_G+1)(3d_G)^{n+m-2}$.
		}
		
		$\qed$
	\end{proof}

	\section{A scalar representation for polymomial matrix inequalities}
	\label{mat:deg}
	In this section,  we provide  exact estimates on the quantity and degree bounds for polynomials  $d_1, \dots, d_t$ with $d_1,\dots,d_t\in \qmod{G}$,  such that the polynomial matrix inequality $G(x)\succeq 0$ can be equivalently described by the scalar polynomial inequalities $d_1\geq 0$, $\dots$, $d_t\geq 0$. This serves as a quantitative version of the following proposition due to Schm{\"u}dgen \cite{sk09}.

	\begin{prop}[Proposition 9, \cite{sk09}]   \label{finitebas}
		Let $G(x)$ be an $m$-by-$m$ symmetric polynomial matrix  and $G(x) \neq 0$. Then there exist $m$-by-$m$ diagonal \rev{polynomial matrices} $D_r$, $m$-by-$m$ symmetric \rev{polynomial matrices} $X_{ \pm r} $ and polynomials $z_r \in$ $\Sigma[x], r=1, \ldots, t$, such that:
		
		\bit
		
		\item[(i)] $X_{+r} X_{-r}=X_{-r} X_{+r}=z_r I, \rev{D_r}=X_{-r} G X_{-r}^T, z_r^2 G=X_{+r} D_r X_{+r}^T$,
		
		\item[(ii)] For $x \in \mathbb{R}^n$, \rev{$G(x) \succeq 0$} if and only if \rev{$D_r(x) \succeq  0$} for all $r=1, \ldots, t$.
		
		\eit			
	\end{prop}
	
	Let
	$ 
	K=\{x\in \mR^n \mid G(x) \succeq  0\}.
	$
	A direct corollary of Proposition \ref{finitebas}  is that there exist finitely many polynomials in  $\qmod{G}$ such that  $K$ can be equivalently described by these polynomials. \rev{We also refer to \cite[Proposition 5]{cj} for a generalization of Proposition \ref{finitebas}, where the author considers  the case that $K$ can be described by possibly infinitely many polynomial matrix inequalities.} Denote
	\[
	\theta(m)=\sum\limits_{i=1}^m\left(\prod_{k=m+1-i}^{m} \frac{k(k+1)}{2}\right). 
	\]
	Recall that
	$
	d_{G}=\max\{\deg(G_{ij}),~1\leq i, j\leq m\}.
	$
	In the following, we provide a quantitative version of Proposition \ref{finitebas}.

	\begin{theorem} \label{thm:scapoly}
		Suppose $G(x)$ is an $m$-by-$m$ symmetric polynomial matrix  with $m\geq 2$. Then there exist  $d_1,d_2,\dots,d_{\theta(m)} \in QM[G]_{3^{m-1}d_{G}}$ such that 
		\[
		K=\{x\in \mR^n\mid d_1(x)\geq0,\dots,d_{\theta(m)}(x)\geq 0 \}.
		\]
		
	\end{theorem}

	\begin{proof}
		\rev{
		We prove the result by applying induction on the matrix size $m$.
		Suppose that  $m=2$ and $G(x)=[G_{ij}(x)]_{i,j=1,2}$. 
		In the following, we show that
		the inequality $G(x)=[G_{ij}(x)]_{i,j=1,2}\succeq 0$ can be equivalently  described as 
		\[
		G_{11}\geq 0,~G_{22}\geq 0,~G_{11}(G_{11}G_{22}-G_{12}^2)\geq 0,
		\]
		\[
		G_{11}+2G_{12}+G_{22}\geq 0,~G_{22}(G_{11}G_{22}-G_{12}^2)\geq 0, 
		\]
		\[
		(G_{11}+2G_{12}+G_{22})((G_{11}+2G_{12}+G_{22})G_{22}-(G_{12}+G_{22})^2)\geq 0.
		\]
		If $G(x)\succeq 0$, we know that $G_{11}(x)\geq 0$, $G_{22}(x)\geq 0$, $G_{11}(x)\cdot G_{22}(x)-G_{12}(x)^2\geq 0$, which imply the above inequalities. Suppose  $x\in \mR^n$ satisfies these inequalities. If $G_{11}(x)\neq 0$ ($G_{22}(x)\neq 0$, respectively), it follows from the inequality $G_{11}(x)\cdot(G_{11}(x)\cdot G_{22}(x)-G_{12}(x)^2)\geq 0$ ~($G_{22}(x)\cdot(G_{11}(x)\cdot G_{22}(x)-G_{12}(x)^2)\geq 0$, respectively) that $G_{11}(x)\cdot G_{22}(x)-G_{12}(x)^2\geq 0$. This implies that $G(x)\succeq 0$. Otherwise, if $G_{11}(x)=G_{22}(x)=0$, the last inequality reduces to $-2G_{12}(x)^3\geq 0$ and the inequality $G_{11}(x)+2G_{12}(x)+G_{22}(x)\geq 0$ reduces to $G_{12}(x)\geq 0$. Then, we have that $G_{12}=0$ and $G(x)=0$.
	}
			\rev{	Note that
		\[
		\bbm
		G_{11}& 0\\
		-G_{12}& G_{11}\\
		\ebm \cdot G \cdot \bbm
		G_{11}& -G_{12}\\
		0& G_{11}\\
		\ebm=\bbm
		G_{11}^3& 0\\
		0& G_{11}(G_{11}G_{22}-G_{12}^2)\\
		\ebm,
		\]
		
		\[
		\bbm
		G_{22}& -G_{12}\\
		0& G_{22}\\
		\ebm \cdot G \cdot  \bbm
		G_{22}& 0\\
		-G_{12}& G_{22}\\
		\ebm=\bbm
		G_{22}(G_{11}G_{22}-G_{12}^2)& 0\\
		0& G_{22}^3\\
		\ebm,
		\]
		\[
		A^T\cdot  G \cdot A
		=\bbm
		(G_{11}+2G_{12}+G_{22})^3& 0\\
		0& (G_{11}+2G_{12}+G_{22})((G_{11}+2G_{12}+G_{22})G_{22}-(G_{12}+G_{22})^2)\\
		\ebm,
		\]
		where
		\[
		A= \bbm
		1& 0\\
		1& 1\\
		\ebm\cdot \bbm
		G_{11}+2G_{12}+G_{22}& -G_{22}-G_{12}\\
		0& G_{11}+2G_{12}+G_{22}\\
		\ebm.
		\]
		We know that  all the above scalar inequalities involve  polynomials  that belong to $\qmod{G}_{3d_{G}}$.
		Note that $\theta(2)=6$. Hence, the result holds for $m=2$. 
	}
		
	\rev{	Suppose the result holds for $m=k$ and  consider the case that $m=k+1$. In the following,  we construct  block diagonal polynomial matrices with a $1\times 1$ block and a $k\times k$ block, such that the inequality $G(x)\succeq 0$ can be equivalently described by the main diagonal blocks of these matrices. This reduction in matrix size allows us to apply induction.}
		For $i\in\{1, \ldots, n\}$,    let $P_{i}$ denote the permutation matrix that swaps the first row and the $i$-th row. For $1\leq i < j \leq k+1$, denote the row operation matrix  $Q^{(ij)}$ as 
		\[
		(Q^{(ij)})_{uv}=\left\{
		\baray{ll}
		1,& \text{if~}(u,v)=(1,1),\dots,(k+1,k+1)\text{~or~if~} (u,v)=(i,j),\\
		0,& \text{otherwise}.\\
		\earay
		\right.
		\]
		\rev{Note that $Q^{(ij)}G(Q^{(ij)})^T$ is the matrix obtained by  first adding the $j$-th row of $G$ to the $i$-th row and then adding the $j$-th column  to the $i$-th column.}
		Let
		\[
		T_{i i}=P_{i}~(i=1,\dots,k+1),~T_{i j}=P_{i}\cdot Q^{(ij)}~(1\leq i < j \leq k+1).
		\]
	\rev{Clearly,  for $i<j$, $	T_{i j} G T_{i j}^\mathrm{T}$ is the matrix obtained by  swapping the first row with the $i$-th row of $Q^{(ij)}G(Q^{(ij)})^T$  and then swapping the first column with  the $i$-th column. Hence, we know that for $1\leq i \leq j \leq k+1$, the  matrices $T_{i j}$    are all invertible and it holds that}
		\[
		T_{i j} G T_{i j}^\mathrm{T} =\left[\begin{array}{cc}
			s_{ij} & (\beta^{(ij)})^\mathrm{T} \\
			\beta^{(ij)} & H^{(ij)}
		\end{array}\right],
		\]	
		where $ H^{(ij)}$is a $k$-by-$k$ symmetric polynomial matrix, $\beta^{(ij)} \in \mR[x]^{k}$ and
		\[
		s_{ii}=G_{ii}~(i=1,\dots,k+1),~	s_{ij}=G_{ii}+G_{jj}+2G_{ij}~(1\leq i < j \leq k+1). 
		\]
		Denote the matrices
		\[
		X_{+}^{(ij)}=\left[\begin{array}{cc}
			s_{ij} &~ 0 \\
			\beta^{(ij)} & ~	s_{ij}\cdot I_{k}
		\end{array}\right],~ X_{-}^{(ij)}=\left[\begin{array}{cc}
			s_{ij} & ~0 \\
			-\beta^{(ij)} & ~	s_{ij}\cdot I_{k}
		\end{array}\right] .
		\]
		By  computations, we have that
		
		\be \label{det:s}
		X_{-}^{(ij)} X_{+}^{(ij)}=X_{+}^{(ij)}X_{-}^{(ij)}=s_{ij}^2 \cdot I_{k+1},
		\ee
		\be \label{det:G}
		(X_{-}^{(ij)}T_{i j})\cdot G\cdot (X_{-}^{(ij)}T_{i j})^{\mathrm{T}}=G^{(ij)},
		\ee 
		
		\be  \label{det:Gij}
		s_{ij}^4T_{i j} G T_{i j}^\mathrm{T}=	X_{+}^{(ij)}\cdot G^{(ij)}	(X_{+}^{(ij)})^{\mathrm{T}},
		\ee
		where 
		\[
		G^{(ij)}=\left[\begin{array}{cc}
			s_{ij}^3 & 0 \\
			0 & s_{ij}\cdot \left(s_{ij} H^{(ij)}-\beta^{(ij)}(\beta^{(ij)})^\mathrm{T} \right)
		\end{array}\right].
		\]	
		Let  
		\[
		B^{(ij)}=s_{ij} \left(s_{ij} H^{(ij)}-\beta^{(ij)}(\beta^{(ij)})^\mathrm{T} \right)~(1\leq i\leq j\leq k+1).
		\]
		Since $T_{ij}$ are scalar matrices,  we have 
		$
		\deg(s_{ij})\leq d_{G},~ \deg(B^{(ij)})\leq 3d_{G}.
		$
		
		In the following, we show that 
		\be \label{equi:K}
		K=\{x\in \mR^n\mid s_{ij}(x)\geq 0, B^{(ij)}(x)\succeq 0,~1\leq i\leq j\leq k+1\}.
		\ee
		If $x\in K$, i.e.,  $G(x) \succeq  0$, it follows from \reff{det:G} that $G^{(ij)}(x)\succeq 0$, which implies that
		\be \label{fir:equK}
		s_{ij}(x)\geq 0, B^{(ij)}(x)\succeq 0,\quad \forall 1\leq i\leq j\leq k+1.
		\ee	
		For the converse part, suppose  $x\in \mR^n$ satisfies \reff{fir:equK}. If $s_{ij}=0$ for all $1\leq i\leq j\leq k+1$, then we have  $G_{ij}(x)=0$ for all $1\leq i,j\leq k+1$, leading to $G(x)=0$. If there exist $1\leq  i_0\leq j_0\leq k+1$ such that $s_{i_0j_0}(x)\neq0$,  the relation \reff{det:Gij} gives that
		\be  \label{det:Gij1}
		G(x) = s_{i_0j_0}^{-4}(x)T_{i j}^{-1}	X_{+}^{(i_0j_0)}\cdot G^{(i_0j_0)}	(X_{+}^{(i_0j_0)})^{\mathrm{T}}(T_{i_0j_0}^\mathrm{T})^{-1}.
		\ee
		Since $ G^{(i_0j_0)}(x)\succeq 0$, we conclude that $G(x)\succeq 0$.

	Note that each  $B^{(ij)}$ is a $k$-by-$k$ symmetric polynomial matrix with  $\deg(B^{(ij)})\leq 3d_{G}$ for $1\leq i\leq j\leq k+1$.	By  induction,    there exist $d_1^{(ij)},\dots,d_{\theta(k)}^{(ij)} \in QM[G]_{3^{k}d_{G}}$ such that 
		\[
		\{x\in \mR^n\mid B^{(ij)}(x)\succeq 0\}=\{x\in \mR^n\mid d_1^{(ij)}(x)\geq 0,\dots,d_{\theta(k)}^{(ij)}(x)\geq 0 \}.
		\]
		Combining this with the relation \reff{equi:K}, we have that
		\[
		K=\{x\in \mR^n\mid s_{ij}(x)\geq 0, d_t^{(ij)}(x)\succeq 0,~1\leq i\leq j\leq k+1,1\leq t\leq \theta(k)\}.
		\] 
		The total number of the above constraints  is 
		$
		\frac{1}{2} (k+1)(k+2)+\frac{1}{2}(k+1)(k+2)\theta(k), 
		$ 
		which equals to $\theta(k+1)$. 
		Hence, we have completed the proof.
		
		$\qed$ 	
	\end{proof}

	~\\
	\noindent
	{\bf Remark.}
	There exist  alternative ways to obtain   finitely many scalar polynomial inequalities that describe $K$.
	For instance, consider the  characteristic polynomial $det[ \lambda I_m-G(x)]$  of $G(x)$. It can be expressed as
	\[
	det[\lambda  I_m-G(x)]=\lambda^m+\sum\limits_{i=1}^m (-1)^ig_i(x)\lambda^{m-i}.
	\]
	Since $G(x)$ is symmetric,  all eigenvalues are real, and  $det[ \lambda I_m-G(x)]$ has only real roots. We know that $G(x)\succeq 0 $ if and only if $g_i(x)\geq 0$ for $i=1,\dots,m$. \rev{Note that $g_1(x)=\operatorname{trace}(G(x))$, which is always in $\qmod{G}$. Hence, if the polynomials $\rev{g_2(x)},\dots,g_m(x) \in \qmod{G}$, the quantity    $\theta(m)$ and the degree $3^{m-1}d_G$ as stated in Theorem \ref{thm:scapoly}   can be improved to $m$ and $md_G$, respectively. However,  $g_2(x),\dots,g_m(x)$ typically do not belong to $\qmod{G}$.}

	\section{Degree bounds for matrix Positivstellens{\"a}tze}
	\label{mat:mat}
	
	This section is mainly to   prove  Theorems \ref{intro:comp:matPut3} and \ref{int:comp:unb}.
	As a byproduct, we can get   a  polynomial bound on the convergence rate of matrix SOS relaxations for solving polynomial matrix optimization, which resolves an open question  raised by Dinh and Pham \cite{dihu}. Let
	\be \label{sec5:K}
	K=\{x\in \mR^n \mid G(x) \succeq  0\}, 
	\ee 
	where $G(x)$  is an  $m$-by-$m$ symmetric polynomial  matrix.

	\subsection{Proof of Theorem \ref{intro:comp:matPut3}}
	For the polynomial matrix $G(x)$, let $d_1,\dots,d_{\theta(m)}$ be the polynomials as in Theorem \ref{thm:scapoly}. Note that we assume 
	\[
	K\subseteq \{x\in\mR^n\mid 1-\|x\|_2^2\geq0\}.
	\]
	Denote
	\be
	\mathcal{G}(x)=-\min \left(\rev{\frac{d_1(x)}{\left\|d_1\right\|_B}}, \ldots, \rev{\frac{d_{\theta(m)}(x)}{\left\|d_{\theta(m)}\right\|_B}},0\right). 
	\ee
	If $\mathcal{G}(x)=0$, we have that  $x\in K$, i.e., $d(x,K)=0$. 
	By Corollary \ref{loj}, there exist positive constants $\eta$, $\kappa$ such that the following {\L}ojasiewicz inequality holds:
	\be  \label{def:Loj:G4}
	d(x,K)^{\eta} \leq \kappa \mathcal{G}(x), ~  \forall x \in \bar{\Delta}.
	\ee
	
	In the following, we give the proof of  Theorem \ref{intro:comp:matPut3}. 
	\begin{theorem}
		\label{comp:matPut3}
		Let $K$ be as in \reff{sec5:K} with $m\geq 2$.	Suppose that $F(x)$ is an  $\ell$-by-$\ell$ symmetric polynomial matrix with $\deg(F)=d$, and $ (1-\|x\|_2^2)\cdot I_{\ell}\in \qmod{G}^{\ell}$.
		If $F(x)\succeq F_{\min} \cdot  I_{\ell} \succ 0$ on $K$,  then $F \in \qmod{G}_k^{\ell}$ for
		\be \label{bound2}
		k\geq C \cdot  {8}^{7\eta}\cdot3^{6(m-1)} \theta(m)^3n^2d_{G}^{6}\kappa^{7}d^{14\eta}\left(\frac{\|F\|_B}{F_{\min}}\right)^{7\eta+3}.
		\ee
		In the above,  $\kappa, \eta$ are respectively the {\L}ojasiewicz constant and exponent  in \reff{def:Loj:G4} and $C>0$ is a constant independent of $F$, $G$. \rev{Furthermore, the {\L}ojasiewicz exponent can be estimated as 
			\[
			\eta\leq (3^{m-1}d_G+1)(3^{m}d_G)^{n+\theta(m)-2}.
			\]
		}
	\end{theorem}

	\rev{
		\begin{proof}
			By Theorem \ref{thm:scapoly}, we know that $d_1,\dots,d_{\theta(m)}\in \qmod{G}_{3^{m-1}d_{G}}$  and 
			\[
			K=\{x\in \mR^n\mid d_1(x)\geq0,\dots,d_{\theta(m)}(x)\geq 0\}.
			\]
			It follows from Theorem \ref{comp:matPut2} that $F \in Q M[D]_k^{\ell}$, where $D$ is the diagonal matrix with diagonal entries $d_1, \ldots, d_{\theta(m)}$, $1-\|x\|_2^2$, for
			\[
			k\geq C \cdot  {8}^{7\eta}\cdot3^{6(m-1)} \theta(m)^3n^2d_{G}^{6}\kappa^{7}d^{14\eta}\left(\frac{\|F\|_B}{F_{\min}}\right)^{7\eta+3}.
			\]
			Equivalently, 	there exist   SOS polynomial matrices  $S_{i}\in \Sigma[x]^{\ell}$ ($i=0,\dots,\theta(m)+1$)  with 
			$	\deg(S_0)\leq k$, $ \deg(d_iS_i)\leq k~(i=1,\dots,\theta(m))$, $	\deg(S_{\theta(m)+1})\leq k-2$,
			such that
			\[
			F=S_0+d_1S_1+\cdots d_{\theta(m)}S_{\theta(m)}+(1-\|x\|_2^2)S_{\theta(m)+1}.
			\]
			Suppose that $S_i=\sum\limits_{j=1}^{s_i}v_{i,j}v_{i,j}^T$ for $v_{i,j}\in \mR[x]_{\lceil k/2\rceil}^{\ell\times 1}$ and $d_i=\sigma_{i}+\sum\limits_{t=1}^{r_i}q_{i,t}^TGq_{i,t}$  for some $\sigma_{i}\in \Sigma[x]_{3^{m-1}d_G}^{1}$, $q_{i,t}\in \mR[x]_{\lceil 3^{m-1}d_G/2-d_G/2\rceil}^{m\times 1}$ . Then, we know that
			\[
			d_iS_i=\sigma_{i}S_i+\sum\limits_{j=1}^{s_i}\sum\limits_{t=1}^{r_i} v_{i,j}q_{i,t}^TGq_{i,t}v_{i,j}^T,
			\]
			which implies that $d_iS_i\in \qmod{G}_{k+3^{m-1}d_G}^{\ell}$.  Since $ (1-\|x\|_2^2)\cdot I_{\ell}\in \qmod{G}^{\ell}$, we know that $1-\|x\|_2^2\in \qmod{G}_{k_0}$ for some $k_0\in \mathbb{N}$. Then, we know that $(1-\|x\|_2^2)S_{\theta(m)+1}\in \qmod{G}_{k+k_0}^{\ell}$. Let $C$ be sufficiently large. Then,  we can drop the small terms $3^{m-1}d_G$, $k_0$, and thus $F \in \qmod{G}_k^{\ell}$ for $k$ satisfying \reff{bound2}.
				Since $d_1,\dots,d_{\theta(m)}\in \qmod{G}_{3^{m-1}d_{G}}$ and 
			\[
			K=\left\{x\in\mR^n \mid \frac{d_1(x)}{\left\|d_1\right\|_B}\geq 0,\dots,\dots,\frac{d_{\theta(m)}(x)}{\left\|d_{\theta(m)}\right\|_B}\geq 0\right\},
			\]
			it follows from Theorem \ref{klpham} that 	$\eta\leq (3^{m-1}d_G+1)(3^{m}d_G)^{n+\theta(m)-2}$.
			
			$\qed$
		\end{proof}
	}

	~\\
	\noindent
	{\bf Remark.}
	When $G(x)$ is a diagonal polynomial matrix  with diagonal entries  $g_1,\dots,g_m \in \mR[x]$, i.e.,
	\be \nonumber
	K=\{x\in \mR^n\mid   g_1(x)\geq 0,\dots,  g_m(x)\geq 0\},
	\ee
	Theorem \ref{comp:matPut3} implies that for $k$ greater than the quantity as in \reff{bound2}, we have $F \in \qmod{G}_k^{\ell}$. However, the bound given in \reff{bound1} from Theorem \ref{comp:matPut2} is generally smaller than that  in \reff{bound2}. Additionally,  the proof of 	Theorem \ref{comp:matPut3}  builds upon Theorem \ref{comp:matPut2}.

	\subsection{Proof of Theorem \ref{int:comp:unb}} \label{aaaaaa}
	
	For a polynomial matrix $F=\sum_{\alpha \in \mathbb{N}^n_{d}}F_{\alpha}x^{\alpha}$ with $\deg(F)=d$, $F^{hom}$  denotes its  highest-degree homogeneous terms, i.e.,
	$
	F^{hom}(x)=\sum_{\alpha \in \mathbb{N}^n,|\alpha|= d}F_{\alpha}x^{\alpha}.
	$ 
	The homogenization of $F$ is denoted by $\widetilde{F}$, i.e.,
	$
	\widetilde{F}(\tilde{x})=\sum_{\alpha \in \mathbb{N}^n_{d}}F_{\alpha}x^{\alpha}x_0^{d-|\alpha|}
	$ 
	for $\tilde{x}:=(x_0,x)$.
	Let
	\be \label{tilde:K}
	\baray{rcl}
	\widetilde{K}  := & \left\{  \tilde{x}\in \mathbb{R}^{n+1}
	\left| \baray{l}
	(x_0)^{d_0}\widetilde{G}(\tilde{x})\succeq 0, \\
	x_0^2 + x^Tx = 1
	\earay \right. \right\}, \\
	\earay
	\ee
	where $d_0=2\lceil d_G/2\rceil-d_G$. Consider the optimization   
	\be  \label{hom:nsdp}
	\left\{ \baray{rl}
	\min & \lmd_{\min}(\widetilde{F}(\tilde{x}))  \\
	\st &  	(x_0)^{d_0}\widetilde{G}(\tilde{x})\succeq 0, \\
	&x_0^2 + x^Tx = 1,\\
	\earay \right.
	\ee
	where $\lmd_{\min}(\widetilde{F}(\tilde{x}))$ is the smallest eigenvalue of  $\widetilde{F}(\tilde{x})$. Let $\widetilde{F}_{\min}$ denote the optimal value of \reff{hom:nsdp}.

	\begin{lemma} \label{lem5.2}
		Let $K$ be as in \reff{sec5:K}.	Suppose that $F(x)$ is an  $\ell$-by-$\ell$ symmetric polynomial matrix with $\deg(F)=d$.
		If $F \succ 0$ on $K$  and $F^{hom}\succ 0$ on $\mR^n$, then   $\widetilde{F}_{\min}> 0$ and $\widetilde{F}\succeq \widetilde{F}_{\min}\cdot I_{\ell}$ on $\widetilde{K}$.
	\end{lemma}
	\begin{proof}
		Suppose $\tilde{x}=(x_0,x)\in \widetilde{K}$.	If $x_0=0$, then we have 
		$
		\widetilde{F}(\tilde{x})=F^{hom}(x)\succ 0.
		$ 
		If $x_0\neq 0$, then
		\[
		G(x/x_0)=(x_0)^{-d_G-d_0}(x_0)^{d_0}\tilde{G}(\tilde{x})\succeq 0,
		\] 
		which implies that  $x/x_0\in K$.	Since $F^{hom}\succ 0$ on $\mR^n$, the degree of $F$ must be  even, and we have
		\[
		\widetilde{F}(\tilde{x})=x_0^dF(x/x_0)\succ 0.
		\] 
		Since  $\widetilde{K}$ is compact, the optimal value of \reff{hom:nsdp} is positive, i.e., $\widetilde{F}_{\min}>0$. 
		
		$\qed$
	\end{proof}

	For the polynomial matrix $(x_0)^{d_0}\widetilde{G}(\tilde{x})$, let $d_1,\dots,d_{\theta(m)}$ be the   polynomials   as in Theorem \ref{thm:scapoly}.
	Denote
	{\small
		\be
		\mathcal{G}(\tilde{x})=-\min \left(\frac{d_1(\tilde{x})}{\left\|d_1\right\|_B}, \ldots, \frac{d_{\theta(m)}(\tilde{x})}{\left\|d_{\theta(m)}\right\|_B}, \frac{1-\|\tilde{x}\|_2^2}{\|1-\|\tilde{x}\|_2^2\|_B},\frac{\|\tilde{x}\|_2^2-1}{\|1-\|\tilde{x}\|_2^2\|_B},0\right). 
		\ee
	}
	Note that $\mathcal{G}(\tilde{x})=0$ implies that $\tilde{x}\in \widetilde{K}$. 
	By Corollary \ref{loj}, there exist positive constants $\eta$, $\kappa$ such that  for all $\tilde{x}$ satisfying $\sqrt{n+1}-x_0-x_1-\cdots-x_n \geq 0,1+x_0\geq 0, 1+x_1 \geq 0, \ldots, 1+x_n \geq 0$, it holds that
	\be  \label{def:Loj:G5}
	d(\tilde{x},\widetilde{K})^{\eta} \leq \kappa \mathcal{G}(\tilde{x}).
	\ee

	In the following, we give the proof of Theorem \ref{int:comp:unb}, which provides a polynomial bound on the complexity of  matrix  Putinar--Vasilescu's Positivstellens{\"a}tz.

	\begin{theorem}
		\label{comp:unb}
		Let $K$ be as in \reff{sec5:K} with $m\geq 2$.	Suppose that $F(x)$ is an  $\ell$-by-$\ell$ symmetric polynomial matrix with $\deg(F)=d$.
		If $F \succ 0$ on $K$  and $F^{hom}\succ  0$ on $\mR^n$,  then $(1+\|x\|_2^2)^k F \in \qmod{G}_{2k+d}^{\ell}$ for
		\be \label{bou:3}
		k\geq C \cdot  {8}^{7\eta}\cdot 3^{6(m-1)} (\theta(m)+2)^3(n+1)^2(\lceil d_G/2\rceil)^{6}\kappa^{7}d^{14\eta}\left(\frac{\|\widetilde{F}\|_B}{\widetilde{F}_{\min}}\right)^{7\eta+3}.
		\ee
		In the above,  $\kappa$, $\eta$ are respectively the {\L}ojasiewicz constant and exponent  in \reff{def:Loj:G5}, $C>0$ is a constant independent of $F$, $G$ and $\widetilde{F}_{\min}$ is the optimal value of \reff{hom:nsdp}.  \rev{Furthermore, the {\L}ojasiewicz exponent can be estimated as 
			\[
			\eta\leq (2\cdot3^{m-1}\lceil d_G/2\rceil+1)(2\cdot 3^{m}\lceil d_G/2\rceil)^{n+\theta(m)-1}.
			\]
		}
	\end{theorem}
	
	\begin{proof}
		By Lemma \ref{lem5.2}, we have that	 $\widetilde{F}\succeq \widetilde{F}_{\min}\cdot I_{\ell}\succ 0$ on $\widetilde{K}$. 
		Note that  $\qmod{1-\|\tilde{x}\|^2}^{\ell}$ is Archimedean, and  $\rev{\widetilde{K}}$ can be equivalently  described by
		\[
		d_1(\tilde{x})\geq 0,\dots,d_{\theta(m)}(\tilde{x})\geq 0,1-\|\tilde{x}\|_2^2\geq 0, \|\tilde{x}\|_2^2-1\geq 0.
		\]  
		By following the proof of    Theorem \ref{intro:comp:matPut3} for  $\widetilde{F}(x)$, we know that  for $k$ satisfying  \reff{bou:3}, 
		there exist  $S\in \Sigma[\tilde{x}]^{\ell}$, $H \in \mathbb{R}[\tilde{x}]^{\ell \times \ell}$, $P_i \in \mathbb{R}[\tilde{x}]^{m\times \ell}$ ($i=1,\dots,t$)  with 
		\[
		\deg(S)\leq k,\, \deg(H)\leq k-2,\,  \deg(P_i^T(x_0)^{d_0}\widetilde{G}P_i)\leq k~(i=1,\dots,t),
		\]	 
		such that
		\[
		\widetilde{F}(\tilde{x})=S(\tilde{x})+(\|\tilde{x}\|_2^2-1)\cdot H(\tilde{x})+\sum\limits_{i=1}^{t} P_i(\tilde{x})^T(x_0)^{d_0}\widetilde{G}(\tilde{x})P_i(\tilde{x}).
		\]
		Replacing  $\tilde{x}$ by $(1,x)/\sqrt{1+\|x\|_2^2}$ in the above identity, we get that
		\be \label{pv1}
		\baray{ll}
		\frac{F(x)}{\sqrt{1+\|x\|_2^2}^d}=&\sum\limits_{i=1}^{t} P_i\left(\frac{(1,x)}{\sqrt{1+\|x\|_2^2}}\right)^T\frac{G(x)}{\sqrt{1+\|x\|_2^2}^{d_0+d_G}}P_i\left(\frac{(1,x)}{\sqrt{1+\|x\|_2^2}}\right)\\
		&+S\left(\frac{(1,x)}{\sqrt{1+\|x\|_2^2}}\right).
		\earay 
		\ee
		Suppose that $S=V^TV$ for some polynomial matrix $V$. Then,  there exist polynomial matrices $V^{(1)}, V^{(2)}$ with $2\deg(V^{(1)})\leq \deg(S)$, $2\deg(V^{(2)})+2\leq \deg(S)$ such that
		\[
		V\left(\frac{(1,x)}{\sqrt{1+\|x\|_2^2}}\right)=\left(\frac{1}{\sqrt{1+\|x\|_2^2}}\right)^{\deg(V)}(V^{(1)}+\sqrt{1+\|x\|_2^2}V^{(2)}).
		\]
		It holds that
		\[
		\baray{ll}
		S\left(\frac{(1,x)}{\sqrt{1+\|x\|_2^2}}\right)=\frac{(V^{(1)})^TV^{(1)}+(1+\|x\|_2^2)(V^{(2)})^TV^{(2)}}{\sqrt{1+\|x\|_2^2}^{\deg(S)}}
		+\frac{(V^{(2)})^TV^{(1)}+(V^{(1)})^TV^{(2)}}{\sqrt{1+\|x\|_2^2}^{\deg(S)-1}}.
		\earay 
		\]
		Similarly, there exist  polynomial matrices $P_i^{(1)}, P_i^{(2)}$ with $\deg(P_i^{(1)})\leq \deg(P_i)$, $\deg(P_i^{(2)})+1\leq \deg(P_i)$ satisfying
		\[
		P_i\left(\frac{(1,x)}{\sqrt{1+\|x\|_2^2}}\right)=\left(\frac{1}{\sqrt{1+\|x\|_2^2}}\right)^{\deg(P_i)}(P_i^{(1)}+\sqrt{1+\|x\|_2^2}P_i^{(2)}),
		\]
		Then, we have that 
		\[
		\baray{l}
		\sum\limits_{i=1}^{t} P_i\left(\frac{(1,x)}{\sqrt{1+\|x\|_2^2}}\right)^T\frac{G(x)}{\sqrt{1+\|x\|_2^2}^{d_0+d_G}}P_i\left(\frac{(1,x)}{\sqrt{1+\|x\|_2^2}}\right)  \\
		\quad\quad\quad\quad\quad\quad	=\sum_{i=1}^{t} \frac{(P_i^{(1)})^TG(x)P_i^{(1)}+(1+\|x\|_2^2)(P_i^{(2)})^TG(x)P_i^{(2)}}{\sqrt{1+\|x\|_2^2}^{d_0+d_G+2\deg(P_i)}} + \frac{(P_i^{(2)})^TG(x)P_i^{(1)}+(P_i^{(1)})^TG(x)P_i^{(2)}}{\sqrt{1+\|x\|_2^2}^{d_0+d_G+2\deg(P_i)-1}}.\\
		\earay
		\]
	Combining these results with  equation \reff{pv1},  we know that for $k$ satisfying \reff{bou:3}, the following holds:
		\[
		(1+\|x\|_2^2)^kF=Q+\sqrt{1+\|x\|_2^2}H(x),
		\]
		for some $Q\in \qmod{G}^{\ell}$ with $\deg(Q)\leq 2k+d$ and a symmetric polynomial matrix $H$. Since $\sqrt{1+\|x\|_2^2}$ is not a rational function, we conclude that $(1+\|x\|_2^2)^kF=Q$, which implies $(1+\|x\|_2^2)^k F\in \qmod{G}^{\ell}_{2k+d}$.
		
	\rev{		Since $d_1,\dots,d_{\theta(m)}\in \qmod{G}_{2\cdot3^{m-1}\lceil d_G/2\rceil}$ and 
		\[
		\widetilde{K}=\left\{\tilde{x}\in\mR^{n+1} \mid \frac{d_1(\tilde{x})}{\left\|d_1\right\|_B}\geq 0,\dots,\frac{d_{\theta(m)}(\tilde{x})}{\left\|d_{\theta(m)}\right\|_B}\geq 0,1-\|\tilde{x}\|_2^2=0\right\},
		\]
		it follows from Theorem \ref{klpham} that 	$\eta\leq (2\cdot3^{m-1}\lceil d_G/2\rceil+1)(2\cdot3^{m}\lceil d_G/2\rceil)^{n+\theta(m)-1}$.
	}
		
		$\qed$
	\end{proof}

\noindent	
{\bf Proof of Corollary \ref{comp:unb2}.}
		Note that  $F\succeq 0$ on $K$, and we have
		\[
		d+2\geq \deg(F+\varepsilon \cdot (1+\|x\|_2^2)^{\lceil \frac{d+1}{2}\rceil}I_{\ell})=2\lceil \frac{d+1}{2}\rceil>d.
		\]
		Hence, for $\tilde{x}\in \widetilde{K}$, the homogenization of $F+\varepsilon \cdot (1+\|x\|_2^2)^{ \lceil \frac{d+1}{2}\rceil}I_{\ell}$ satisfies
		\[
		\baray{ll}
		x_0^{2\lceil \frac{d+1}{2}\rceil-d}\widetilde{F}+\varepsilon \cdot (x_0^2+\|x\|_2^2)^{\lceil \frac{d+1}{2}\rceil}I_{\ell} \succeq \varepsilon \cdot (x_0^2+\|x\|_2^2)^{\lceil \frac{d+1}{2}\rceil}I_{\ell}
		\succeq \varepsilon \cdot I_{\ell}.\\
		\earay
		\] 	
		The conclusion follows  from Theorem \ref{comp:unb}. 
		
		$\qed$

	\subsection{Convergence rate of matrix SOS relaxations}
	\label{sec:momrel}

	Consider the polynomial matrix 
	optimization 
	\be  \label{nsdp}
	\left\{ \baray{rl}
	\min & f(x)  \\
	\st &  	G(x)\succeq 0, \\
	\earay \right.
	\ee
	where $f(x)\in \mR[x]$, $G(x)$ is an $m\times m$ symmetric  polynomial matrix. We denote the optimal value of \reff{nsdp} as $f_{\min}$.

	A standard approach for solving \reff{nsdp} globally is the sum-of-squares hierarchy introduced in \cite{hdlb,kojm,schhol}. This hierarchy consists of  a sequence of semidefinite programming  relaxations.
	Recall that for a degree $k$, the $k$th degree truncation of $\qmod{G}$ is
	\be \nonumber 
	\qmod{G}_{k}:=\left\{\baray{l|l}
	\sigma+\sum_{i=1}^t v_i^TGv_i& \baray{l}  \sigma \in \Sigma[x],~v_i \in \mathbb{R}[x]^{m},~t\in \mathbb{N},\\
	\deg(\sigma)\leq k,~\deg(v_i^TGv_i)\leq k
	\earay
	\earay \right\}.
	\ee
	Checking whether a polynomial belongs to the  the quadratic module $\qmod{G}_{k}$  can
	be done by solving a  semidefinite program (see \cite{schhol}). 
	For an order $k$, the $k$th order SOS   relaxation of \reff{nsdp} is
	\be \label{sos}
	\left\{\begin{array}{cl}
		\max & \gamma \\
		\text { s.t. } & f-\gamma \in \mathrm{QM}[G]_{ 2k}.
	\end{array}\right.
	\ee
	Let $ f_k$ denote the optimal value of \reff{sos}. It  holds  the monotonicity relation:
	$
	f_k \leq f_{k+1} \leq \cdots \leq f_{\min} .
	$
	The hierarchy \reff{sos} is said to have {\it asymptotic convergence} if $f_k\rightarrow \infty$ as $k\rightarrow \infty$, and it is said to have {\it finite convergence} if $f_k=f_{\min}$ for all $k$ big enough.  
	\rev{When the quadratic module $\qmod{G}$ is Archimedean}, it was  shown in \cite{hdlb,kojm,schhol} that the hierarchy \reff{sos}   converges asymptotically.  Recently, Huang and Nie \cite{hny24} proved that this hierarchy also has finite convergence if, in addition, the nondegeneracy condition, strict complementarity condition and second order sufficient condition hold at every global minimizer of \reff{nsdp}.  In the following, we provide a polynomial bound on  the convergence rate of the hierarchy \reff{sos}.

	\begin{theorem} \label{matrate}
		Suppose  $f \in \mathbb{R}[x]$ with $d=\deg(f)$, $m\geq 2$, and $ 1-\|x\|_2^2\in \qmod{G}$.		Then we have that 
		\be \label{rate} 
		|f_k-f_{\min}|\leq 3 \cdot (C \cdot  {8}^{7\eta}\cdot 3^{6(m-1)} \theta(m)^3n^2d_{G}^{6}\kappa^{7}d^{14\eta})^{\frac{1}{7\eta+3}}\|f\|_B\left(\frac{1}{k}\right)^{\frac{1}{7\eta+3}},
		\ee
		where  $\kappa$, $\eta$ are respectively the {\L}ojasiewicz constant and exponent  in \reff{def:Loj:G4}, and $C>0$ is a constant independent of $f$, $G$. \rev{Furthermore, the {\L}ojasiewicz exponent can be estimated as 
			\[
			\eta\leq (3^{m-1}d_G+1)(3^{m}d_G)^{n+\theta(m)-2}.
			\]
		}
	\end{theorem}

	\begin{proof}
		Since $f_{\min}\leq \|f\|_B$ (see \cite{gefa}), we have that for  $0<\varepsilon\leq \|f\|_B$,
		\[
		\|f-f_{\min}+\varepsilon\|_B\leq \|f\|_B+|f_{\min}|+\varepsilon \leq 3\|f\|_B.
		\]
		Since $f-f_{\min}+\varepsilon \geq \varepsilon $ on $K$,	it follows from Theorem \ref{intro:comp:matPut3} that $f-f_{\min}+\varepsilon \in \qmod{G}_k$ for
		$$	
		k	\geq C \cdot  {8}^{7\eta} \cdot 3^{6(m-1)} \theta(m)^3n^2d_{G}^{6}\kappa^{7}d^{14\eta} \left(\frac{3\|f\|_B}{\varepsilon}\right)^{7 \eta+3}, \\
		$$
		and $\eta\leq (3^{m-1}d_G+1)(3^{m}d_G)^{n+\theta(m)-2}$.
		Then, we know that $f_{\min}-\varepsilon $ is feasible for  \reff{sos}, i.e., $f_k\geq f_{\min}-\varepsilon$ when  $k$ satisfies the above inequality. 
		Equivalently, we know that the relation $\reff{rate}$ holds.

		$\qed$
	\end{proof}

	\section{Conclusions and discussions}
	\label{sec:dis}
	
In this paper, we establish a polynomial bound on the degrees of terms appearing in matrix Putinar's Positivstellens{\"a}tz. For the special case that  the constraining polynomial matrix is diagonal, this result improves the exponential bound shown in \cite{helnie}.   As a byproduct, we also obtain a polynomial bound on the convergence rate of matrix SOS relaxations,  which
	resolves an open question posed by Dinh and  Pham \cite{dihu}.  When the constraining set is unbounded,  a similar bound is provided for   matrix  Putinar--Vasilescu's Positivstellens{\"a}tz.
	
	Theorem \ref{intro:comp:matPut3} uses the {\L}ojasiewicz  inequality \reff{def:Loj:G4} for  equivalent scalar  polynomial inequalities of $K$.  There also exists the  {\L}ojasiewicz  inequality for  polynomial matrices (see \cite{dihu1,dihu}), where the corresponding {\L}ojasiewicz exponent is typically smaller than the  scalar one as in \reff{def:Loj:G4}. An interesting future work is to provide a  proof  based directly on the {\L}ojasiewicz  inequality for $G(x)$, which may yields a better bound.  
	On the another hand,  the set  $\{x\in \mR^n\mid G(x)\succeq 0\}$ can be equivalently described by the quantifier set
	$
	\{x\in \mR^{n} \mid y^TG(x)y\geq 0, ~\forall~ \|y\|_2=1\}.
	$
	It would also be interesting to study the complexity of  Putinar-type Positivstellens{\"a}tze with universal quantifiers for the above reformulation. We refer to \cite{huin} for the recent work on  Positivstellens{\"a}tze with universal quantifiers.

	When $K$ is given by scalar polynomial inequalities, it was shown in  \cite{Robinson} that the {\L}ojasiewicz exponent $\eta$ as in \reff{Lojexp2} is equal to $1$ if the   linear independence constraint qualification  holds at every $x\in K$. Consequently, we can get a representation with order $\left(\frac{\|F\|_B}{F_{\min}}\right)^{10}$. An analogue to the linear independence
	constraint qualification  in nonlinear  semidefinite programming  is  the nondegeneracy condition \cite{shap,sunde}. It would also be interesting to explore whether similar results exist for polynomial matrix inequalities under the nondegeneracy condition.

	~\\
	{\bf Acknowledgements:} The author would like to thank Prof. Konrad Schm\"{u}dgen  for helpful discussions on Proposition \ref{finitebas},  Prof. Tien-Son Pham for fruitful comments on  this paper.
	The author also thanks the associate editor and three anonymous referees for valuable comments and suggestions.


\end{document}